\documentclass[11pt]{amsart}
\usepackage[normalem]{ulem} 
\usepackage{lmodern}
\usepackage{amsmath}
\usepackage{amssymb}
\usepackage{dsfont}  
\usepackage{hyperref}
\usepackage{graphicx}
\usepackage[all]{xy}
\usepackage{latexsym}
\usepackage{epstopdf}
\usepackage{pgf}
\usepackage{tikz}
\usepackage{tikz-cd}
\usepackage{pgf}
\usepackage{subcaption} 
\usepackage{float}

\usetikzlibrary{matrix,arrows,backgrounds,shapes.misc,shapes.geometric,patterns,calc,positioning}
\usetikzlibrary{calc,shapes}
\usetikzlibrary{decorations.pathmorphing}

\usepackage[left=2.8cm,right=2.8cm,top=2.8cm,bottom=2.8cm]{geometry}

\theoremstyle{theorem}
\newtheorem{theorem}{Theorem}[section]

\newtheorem{thm}{Theorem}
        
\newtheorem{lemma}[theorem]{Lemma}
\newtheorem{cor}[theorem]{Corollary}
   
\newtheorem{prop}[theorem]{Proposition}   
\newtheorem{conj}[theorem]{Conjecture}

\theoremstyle{remark}
\newtheorem{remark}[theorem]{Remark}
\newtheorem{notation}[theorem]{Notation}
\newtheorem*{rem}{Remark}
\newtheorem{ex}[theorem]{Example}

\theoremstyle{definition}
\newtheorem{defi}[theorem]{Definition} 

\newcommand{\rk}{\operatorname{rank}\nolimits}
\newcommand{\rkk}{\operatorname{rk}\nolimits}
\newcommand{\len}{\operatorname{len}\nolimits}
\newcommand{\rad}{\operatorname{rad}\nolimits}
\newcommand{\lcm}{\operatorname{lcm}\nolimits}

\newcommand{\Hom}{\operatorname{Hom}\nolimits}
\newcommand{\Ext}{\operatorname{Ext}\nolimits}
\newcommand{\End}{\operatorname{End}\nolimits}
\newcommand{\uCM}{\underline{\mathrm{CM}}\hspace{1pt}}
\newcommand{\CM}{\mathrm{CM}\hspace{1pt}}
\newcommand{\ind}{\operatorname{ind}\nolimits}

\newcommand{\C}{\mathbb{C}}
\newcommand{\B}{\operatorname{B}_{k,n}}
\newcommand{\Sub}{\mathrm{Sub}\hspace{1pt}}
\newcommand{\uSub}{\underline{\mathrm{Sub}}\hspace{1pt}}
\newcommand{\mmod}{\mathrm{mod}\hspace{1pt}}

\newcommand{\Cc}{\mathcal{C}}

\newcommand{\pr}{\mathrm{prf}}

\newcommand{\dimv}{\underline{\mathrm{dim}}\hspace{1pt}}

\newcommand{\thfrac}[3]{ 	
\begin{aligned} #1\\ \hline \\[-3\jot] #2 \\ \hline \\[-3\jot] #3 \end{aligned}
}

\newcommand{\ffrac}[4]{ 	
\begin{aligned} #1\\ \hline \\[-3\jot] #2 \\ \hline \\[-3\jot] #3 \\ \hline \\[-3\jot] #4 \end{aligned}
}

\usepackage{color}

\title{Cluster categories from grassmannians and root combinatorics}
\author{Karin Baur, Dusko Bogdanic, Ana Garcia Elsener}

\newcommand{\monthword}[1]{\ifcase#1\or January\or February\or March\or April\or May\or 
June\or July\or August\or September\or October\or November\or December\fi}
\date{\monthword{\the\month} \the\day, \the\year } 

\begin{document}

\maketitle 

\begin{abstract}

The category of Cohen-Macaulay modules of an algebra $\B$ is used 
in \cite{JKS16} to give an additive categorification of the cluster algebra structure 
on the homogeneous coordinate ring of the Grassmannian of $k$-planes in $n$-space. 
In this paper, we find canonical Auslander--Reiten sequences and 
study the Auslander--Reiten translation periodicity for this category. Furthermore, we give an explicit 
construction of Cohen-Macaulay modules of arbitrary rank. We then use our results to establish a correspondence between 
rigid indecomposable modules of rank 2 and real 
roots of degree 2 for the associated Kac-Moody algebra in the tame cases.
\end{abstract}

%
\section{Introduction}\label{sec:intro}

In this paper we investigate the category $\CM (\B)$ of Cohen-Macaulay modules of an algebra $\B$, defined 
to categorify the cluster structure of the Grassmannian coordinate rings of $k$-planes in $n$-space $\mathbb{C}[G_{k,n}]$. 
We study parts of the Auslander-Reiten quiver of this
category, mostly those Auslander-Reiten sequences and components containing rigid modules.  Of particular interest are the indecomposable
Cohen-Macaulay rigid modules with the same class in the Grothendieck group, i.e.\ the modules with
the same rank 1 modules appearing as composition factors in their filtrations. 
In \cite[Section 5]{JKS16} the authors classify rank 1 modules by associating to each $k$-subset $I$ a module $L_I$, 
and they prove that such modules are in one-to-one
correspondence with the Pl{\"u}cker coordinates $\Phi_I$ in $\mathbb{C}[G_{k,n}]$.

Our first main result is a construction of canonical Auslander--Reiten sequences with rank 1 modules as 
end terms. A $k$-subset of $\mathbb{Z}_n$ is {\em almost consecutive} if it is a union of two intervals, 
one of them consisting of a single entry.

\begin{thm}[Theorem~\ref{Teo rk1-new}] \label{thm:AR-sequence}
Let $I$ be an almost consecutive $k$-subset and $J$ be such that 
$L_J=\Omega(L_I)$. Then there is an AR-sequence 
\[
L_I \to L_X/L_Y\to L_J. 
\] 
The middle 
term is a rigid rank $2$ module which is indecomposable if and only if $X$ (and $Y$) are almost 
consecutive. 
\end{thm}

It is not hard to see, given the triangulated structure and the Auslander--Reiten  translation of $\CM(\B)$, that non-projective rank 1 
modules are $\tau$ periodic (cf.~\cite[Proposition 2.7]{BB}). 
We use results from Demonet-Luo \cite{DL16} to explain that this category is 
$\tau$ periodic (see Section~\ref{sec:periodic}).

In the spirit of the construction of rank 1 modules, we give a definition of modules of arbitrary rank in 
Section~\ref{sec:higher-rank}, and we prove in Proposition~\ref{prop:higher-rank} that such modules are free 
over the center of $\B$, hence they are Cohen--Macaulay. 

It was observed in \cite[Section 8]{JKS16} that, when $\CM(\B)$ is of finite representation type, there is a correspondence between indecomposable rank $d$ modules and roots of degree $d$ for an associated 
Kac-Moody algebra. 
We study this correspondence for rank 2 modules  when $\CM(\B)$ is of tame representation type. Our main tool for this are 
Auslander-Reiten quivers. We compute the tubular components of the Auslander--Reiten quiver containing rigid modules of rank 1 and rank 2. We show that for each rigid indecomposable module of rank 2 
we can cycle the filtration layers to obtain a new rigid indecomposable module. This yields the next result.

\begin{thm}[Section~\ref{sec:count-all}, Section~\ref{sec:count-all-48}] \label{te:cyclemods}
When $(k,n)=(3,9)$ or $(k,n)=(4,8)$, for every real root of 
degree $2$ there are two rigid indecomposable modules. 
Moreover, if $M$ is such a rigid indecomposable rank $2$ module 
and if its filtration by rank $1$ modules is $L_I|L_J$, then the rank $2$ module with 
filtration $L_J|L_I$ is also rigid indecomposable.
\end{thm}

Furthermore, in Conjecture~\ref{conj:counting}, we give the number of rigid indecomposable 
rank 2 modules corresponding to 
real roots in the general case.

%
\section{Background}\label{sec:background}


In the following we recall the definition of $\B$, the category of Cohen-Macaualy modules $\CM(\B)$, and the relation between this category and root systems.

\subsection{The category $\CM (\B)$} We follow the exposition from \cite{BKM16} in order to introduce notation. 
Let $n$ and $k$ be integers such that  $1<k\leq\frac{n}{2}$. 
Let $C$ be  a circular graph  with vertices $C_0=\mathbb Z_n$ set clockwise around a
circle, and with the set of edges, $C_1$, also labeled by $\mathbb Z_n$, with edge $i$ joining vertices $i-1$ and $i$. For integers $a,b\in \{1,2,\dots, n\}$, we denote by $[a,b]$ the closed cyclic
interval consisting of the elements of the set $\{a, a +1, \dots, b\}$ reduced modulo $ n$.
Consider the quiver with vertices $C_0$
and, for each edge $i\in C_1$, a pair of arrows $x_i{\colon}i-1\to i$
and $y_i{\colon}i\to i-1$. Then we consider the quotient of the
path algebra over $\mathbb{C}$ of this quiver by the ideal generated by the
$2n$ relations $xy=yx$ and $x^k=y^{n-k}$. Here, we  interpret $x$ and $y$
as arrows of the form $x_i,y_i$ appropriately, and starting at any vertex. For example,  when $n=5$ we have the quiver

\begin{center}
\begin{tikzpicture}[scale=1]
\newcommand{\radius}{1.5cm}
\foreach \j in {1,...,5}{
  \path (90-72*\j:\radius) node[black] (w\j) {$\bullet$};
  \path (162-72*\j:\radius) node[black] (v\j) {};
  \path[->,>=latex] (v\j) edge[black,bend left=30,thick] node[black,auto] {$x_{\j}$} (w\j);
  \path[->,>=latex] (w\j) edge[black,bend left=30,thick] node[black,auto] {$y_{\j}$}(v\j);
}
\draw (90:\radius) node[above=3pt] {$5$};
\draw (162:\radius) node[above left] {$4$};
\draw (234:\radius) node[below left] {$3$};
\draw (306:\radius) node[below right] {$2$};
\draw (18:\radius) node[above right] {$1$};
\end{tikzpicture}
\end{center}

The completion $\B$ of this algebra coincides with the quotient of the completed path
algebra of the graph $C$, i.e.\ the doubled quiver as above,
by the closure of the ideal generated by the relations above (we view the completed path
algebra of the graph $C$
as a topological algebra via the $m$-adic topology, where $m$ is the two-sided ideal
generated by the arrows of the quiver, see \cite[Section 1]{DWZ08}). The algebra 
$\B$, that we will often denote by $B$ when there is no ambiguity, 
was introduced in \cite{JKS16}, Section 3.
Observe that $\B$ is isomorphic to $\mathrm{B}_{n-k,n}$, so we will always take $k\le \frac n 2$. 

\smallskip

The center $Z$ of $B$ is the ring of formal power series $\mathbb{C}[[t]]$,
where $t=\sum_{i=1}^n x_iy_i$.
The (maximal) Cohen-Macaulay $B$-modules are precisely those which are
free as $Z$-modules. Indeed, such a module $M$ is given by a representation
$\{M_i\,:\,i\in C_0\}$ of
the quiver with each $M_i$ a free $Z$-module of the same rank
(which is the rank of $M$, cf. \cite{JKS16}, Section 3).

\begin{defi}[\cite{JKS16}, Definition 3.5]
For any $B$-module $M$, if $K$ is the field of fractions of $Z$, we define its \emph{rank}
\[
 \rkk (M) = {\len}\bigl( M\otimes_Z K \bigr).
\]
\end{defi}

Note that $B\otimes_Z K\cong M_n ( K)$, 
which is a simple algebra. It is easy to check that the rank is additive on short exact sequences,
that $\rkk (M) = 0$ for any finite-dimensional $B$-module 
(because these are torsion over $Z$) and 
that, for any Cohen-Macaulay $B$-module $M$ and every idempotent $e_j$, $1\leq j\leq n$, $\rkk_Z(e_j M) = \rkk(M)$, so that, in particular, $\rkk_Z(M) = n  \rkk(M)$.

\begin{defi}[\cite{JKS16}, Definition 5.1] \label{d:moduleMI}
For any $k$-subset   $I$  of $C_1$, we define a rank $1$ $B$-module
\[
  L_I = (U_i,\ i\in C_0 \,;\, x_i,y_i,\, i\in C_1)
\]
as follows.
For each vertex $i\in C_0$, set $U_i=\mathbb C[[t]]$ and,
for each edge $i\in C_1$, set
\begin{itemize}
\item[] $x_i\colon U_{i-1}\to U_{i}$ to be multiplication by $1$ if $i\in I$, and by $t$ if $i\not\in I$,
\item[] $y_i\colon U_{i}\to U_{i-1}$ to be multiplication by $t$ if $i\in I$, and by $1$ if $i\not\in I$.
\end{itemize}
\end{defi}

The module $L_I$ can be represented by a lattice diagram
$\mathcal{L}_I$ in which $U_0,U_1,U_2,\ldots, U_n$ are represented by columns from
left to right (with $U_0$ and $U_n$ to be identified).
The vertices in each column correspond to the natural monomial 
$\mathbb C$-basis of $\mathbb C[t]$.
The column corresponding to $U_{i+1}$ is displaced half a step vertically
downwards (respectively, upwards) in relation to $U_i$ if $i+1\in I$
(respectively, $i+1\not \in I$), and the actions of $x_i$ and $y_i$ are
shown as diagonal arrows. Note that the $k$-subset $I$ can then be read off as
the set of labels on the arrows pointing down to the right which are exposed
to the top of the diagram. For example, the lattice picture $\mathcal{L}_{\{1,4,5\}}$
in the case $k=3$, $n=8$, is shown in the following picture   

\begin{figure}[H]
\begin{tikzpicture}[scale=0.8,baseline=(bb.base),
quivarrow/.style={black, -latex, thin}]
\newcommand{\seventh}{51.4} 
\newcommand{\circradius}{1.5cm}
\newcommand{\inradius}{1.2cm}
\newcommand{\outradius}{1.8cm}
\newcommand{\dotrad}{0.1cm} 
\newcommand{\bdrydotrad}{{0.8*\dotrad}} 
\path (0,0) node (bb) {}; 


\draw (0,0) circle(\bdrydotrad) [fill=black];
\draw (0,2) circle(\bdrydotrad) [fill=black];
\draw (1,1) circle(\bdrydotrad) [fill=black];
\draw (2,0) circle(\bdrydotrad) [fill=black];
\draw (2,2) circle(\bdrydotrad) [fill=black];
\draw (3,1) circle(\bdrydotrad) [fill=black];
\draw (3,3) circle(\bdrydotrad) [fill=black];
\draw (4,0) circle(\bdrydotrad) [fill=black];
\draw (4,2) circle(\bdrydotrad) [fill=black];
\draw (5,1) circle(\bdrydotrad) [fill=black];
\draw (6,0) circle(\bdrydotrad) [fill=black];
\draw (6,2) circle(\bdrydotrad) [fill=black];
\draw (7,1) circle(\bdrydotrad) [fill=black];
\draw (7,3) circle(\bdrydotrad) [fill=black];
\draw (8,2) circle(\bdrydotrad) [fill=black];
\draw (8,4) circle(\bdrydotrad) [fill=black];
\draw (8,0) circle(\bdrydotrad) [fill=black];


\draw [quivarrow,shorten <=5pt, shorten >=5pt, ultra thick] (0,2)-- node[above]{$1$} (1,1);
\draw [quivarrow,shorten <=5pt, shorten >=5pt] (1,1) -- node[above]{$1$} (0,0);
\draw [quivarrow,shorten <=5pt, shorten >=5pt, ultra thick] (2,2) -- node[above]{$2$} (1,1);
\draw [quivarrow,shorten <=5pt, shorten >=5pt] (1,1) -- node[above]{$2$} (2,0);
\draw [quivarrow,shorten <=5pt, shorten >=5pt, ultra thick] (3,3) -- node[above]{$3$} (2,2);
\draw [quivarrow,shorten <=5pt, shorten >=5pt] (2,2) -- node[above]{$3$} (3,1);
\draw [quivarrow,shorten <=5pt, shorten >=5pt] (3,1) -- node[above]{$3$} (2,0);
\draw [quivarrow,shorten <=5pt, shorten >=5pt, ultra thick] (3,3) -- node[above]{$4$} (4,2);
\draw [quivarrow,shorten <=5pt, shorten >=5pt] (4,2) -- node[above]{$4$} (3,1);
\draw [quivarrow,shorten <=5pt, shorten >=5pt] (3,1) -- node[above]{$4$} (4,0);
\draw [quivarrow,shorten <=5pt, shorten >=5pt, ultra thick] (4,2) -- node[above]{$5$} (5,1);
\draw [quivarrow,shorten <=5pt, shorten >=5pt] (5,1) -- node[above]{$5$} (4,0);
\draw [quivarrow,shorten <=5pt, shorten >=5pt, ultra thick] (6,2) -- node[above]{$6$} (5,1);
\draw [quivarrow,shorten <=5pt, shorten >=5pt] (5,1) -- node[above]{$6$} (6,0);
\draw [quivarrow,shorten <=5pt, shorten >=5pt] (6,2) -- node[above]{$7$} (7,1);
\draw [quivarrow,shorten <=5pt, shorten >=5pt] (7,1) -- node[above]{$7$} (6,0);
\draw [quivarrow,shorten <=5pt, shorten >=5pt, ultra thick] (7,3) -- node[above]{$7$} (6,2);
\draw [quivarrow,shorten <=5pt, shorten >=5pt] (7,3) -- node[above]{$8$} (8,2);
\draw [quivarrow,shorten <=5pt, shorten >=5pt] (8,2) -- node[above]{$8$} (7,1);
\draw [quivarrow,shorten <=5pt, shorten >=5pt, ultra thick] (8,4) -- node[above]{$8$} (7,3);
\draw [quivarrow,shorten <=5pt, shorten >=5pt] (7,1) -- node[above]{$8$} (8,0);

\draw [dotted] (0,-2) -- (0,2);
\draw [dotted] (8,-2) -- (8,4);

\draw [dashed] (4,-2) -- (4,-1);
\end{tikzpicture}
\caption{Lattice diagram of the module $L_{\{1,4,5\}}$} \label{Lattice}
\end{figure}
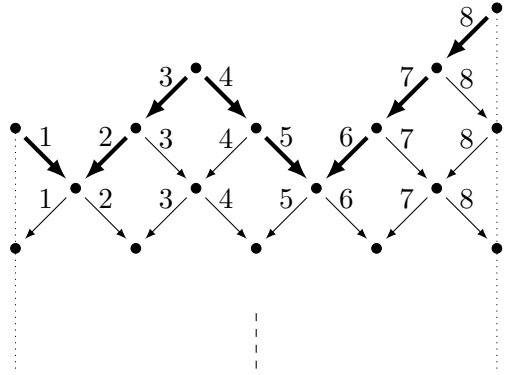

\vspace{2mm}
We see from Figure \ref{Lattice} that the module $L_I$ is determined by its upper boundary, denoted  by the thick lines, 
which we refer to as the rim of the module $L_I$ (this is why call the $k$-subset $I$ as the rim of $L_I$). 
Throughout this paper we will identify a rank 1 module $L_I$ with its rim. Moreover, most of the time we will omit the arrows in the rim of $L_I$ and represent it as an undirected graph. 

We say that $i$ is a {\em peak} of the rim $I$ if $i\notin I$ and $i+1\in I$.  In the above example, the peaks of $I=\{1,4,5\}$ are $3$ and $8$. 

\begin{rem} {\rm
We identify the end points of a rim $I$. Unless specified otherwise, we will assume that the leftmost vertex is  labeled by $n$, and in this case, we may omit labels on the edges of the rim. Looking from left to right, the number of  downward edges in the rim is $k$ (these are the edges labeled by the elements of $I$), and the number of upward edges is  $n-k$ (these are the edges labeled by the elements of $[1,n] \setminus I$). }
\end{rem}

\begin{prop}[\cite{JKS16}, Proposition 5.2]
Every rank $1$ Cohen-Macaulay $B$-module is isomorphic to $L_I$
for some unique $k$-subset $I$ of $C_1$.
\end{prop}

Every $B$-module has a canonical endomorphism given by multiplication by $t\in Z$.
For ${L}_I$ this corresponds to shifting $\mathcal{L}_I$ one step downwards.
Since $Z$ is central, $\Hom_B(M,N)$ is
a $Z$-module for arbitrary $B$-modules $M$ and $N$.

If $M,N$ are free $Z$-modules, then so is $\Hom_B(M,N)$. In particular, for rank 1 
Cohen-Macaulay $B$-modules $L_I$ and $L_J$, $\Hom_B(L_I,L_J)$ is a free 
module of rank 1 
over $Z=\mathbb C[[t]]$, generated by the canonical map given by placing the 
lattice of $L_I$ inside the lattice of $L_J$ as far up as possible so that no part of the rim of $L_I$ is strictly above the rim of $L_J$ \cite[Section 6]{BKM16}.


The algebra $B$ has $n$ indecomposable projective left modules $P_j=Be_j$, corresponding to the vertex idempotents $e_j\in B$, for $j\in C_0$.
Our convention is that representations of the quiver correspond to 
left  $B$-modules. The projective indecomposable $B$-module $P_j$ is the rank 1 module $L_I$, where $I=\{j+1, j+2, \dots, j+k\}$, so we represent projective indecomposable modules as in the following picture, where $P_5$ is pictured ($n=5$, $k=3$):
\begin{center}
\begin{tikzpicture}[scale=0.8]
\foreach \j in {0,...,5}
  {\path (\j,3.5) node (a) {$\j$};}
\path (0,3) node (a0) {$\bullet$}; \path (5,2) node (a5) {$\bullet$}; 
\foreach \v/\x/\y in
  {a1/1/2, a2/2/1, a3/3/0, a4/4/1, b0/0/1, b1/1/0, b2/2/-1, b4/4/-1, b5/5/0, c0/0/-1}
  {\path (\x,\y) node (\v) {$\bullet$};}
\foreach \j in {1,3,5}
  {\path (\j,-1.5) node {$\vdots$};}
\foreach \t/\h in
  {a0/a1, a1/a2, a2/a3, b0/b1, b1/b2, a3/b4, a4/b5}
  {\path[->,>=latex] (\t) edge[black,thick] node[black,above right=-2pt] {$x$} (\h);}
\foreach \t/\h in
  {a4/a3, a5/a4, b5/b4, a3/b2, a2/b1, a1/b0, b1/c0}
  {\path[->,>=latex] (\t) edge[black,thick] node[black,above left =-3pt] {$y$} (\h);}
\end{tikzpicture}
\end{center}

\begin{defi}
A pair $I, J$ of $k$-subsets of $C_1$ is said to be non-crossing (or weakly separated) if there are no elements $a, b, c, d$, cyclically ordered around $C_1$, such that $a,c \in I \setminus J$ and $b,d \in J \setminus I$.
\end{defi}

\begin{defi}\label{def-rigid}
A $B$-module is \emph{rigid} if $\Ext^1_B (M,M)=0$.
\end{defi}

If $I$ and $J$ are non-crossing $k$-subsets, then $\Ext_{B}^1(L_I,L_J)=0$, in particular, rank 1 modules are rigid (see \cite[Proposition 5.6]{JKS16}).

\begin{notation} 
Every rigid indecomposable $M$ of rank $n$ in $\CM (B)$ has a filtration having factors 
$L_{I_1},L_{I_{2}},\dots, L_{I_n}$ of rank 1. 
This filtration is noted in its \emph{profile}, 
$\pr (M) = I_1 | I_2|\ldots | I_n$, \cite[Corollary 6.7]{JKS16}.
\end{notation}


The category $\CM (B)$ provides a categorification for the cluster structure of 
Grassmannian coordinate rings. As we will discuss later, the stable category $\uCM (B)$ 
is 2-Calabi-Yau. Maximal non-crossing collections of $k$-subsets give rise to cluster-tilting 
objects $T$ as the corresponding rank 1 modules are pairwise ext-orthogonal. 
Given a maximal collection of non-crossing $k$-sets $\mathcal{I}$ (including the projectives, 
i.e.\  $k$-sets consisting of a single interval), 
the direct sum 
$T= \oplus_{I \in \mathcal{I}} L_{I}$ corresponds to an 
alternating strand diagram  
\cite{Postnikov} 
whose associated quiver is 
an example of a \emph{dimer model with boundary} \cite[Section 3]{BKM16}. 
If we forget its frozen vertices (the vertices corresponding to projective indecomposables) 
we obtain a 
quiver with potential $(Q,P)$ 
encoding the endomorphism algebra $\End_{\uCM} (T)$ as a finite-dimensional 
Jacobian algebra $J(Q,P)$ in the sense of \cite{DWZ08}. 

\begin{remark}\label{rem:cto-exists}
Any given $k$-subset $I$ can be completed to a maximal non-crossing 
collection $\mathcal{I}$. The arrows in the quiver $Q$ of 
$\End_{\uCM} (T)$ represent morphims in $\Hom_{\uCM} (L_I,L_J)$ that do not factor  through $L_U$ with $U \in \mathcal{I}$. 
Note that the quiver $Q$  has no loops. 
\end{remark}


%
\subsection{Root combinatorics}\label{sec:roots}

Here we recall the connection between indecomposable modules of  $\CM(\B)$ for 
$k$ and $n$ as above and 
roots for an associated Kac-Moody algebra, as explained in~\cite{JKS16}. 

For $(k,n)$ let $J_{k,n}$ be the tree obtained by drawing a Dynkin diagram of 
type $A_{n-1}$, labeling the nodes $1,2,\dots, n-1$ and adding a node $n$ with 
an edge to node $k$. We consider positive roots for 
the associated Kac-Moody algebra, denoting the simple root associated with node $i$ 
by $\alpha_i$ 
for $i=1,\dots, n-1$ and the simple root associated with $n$ by $\beta$. 
For $k=2$, the resulting diagram $J_{k,n}$ is a Dynkin diagram of type 
$D_n$. 
\begin{center}
\includegraphics[height=1cm]{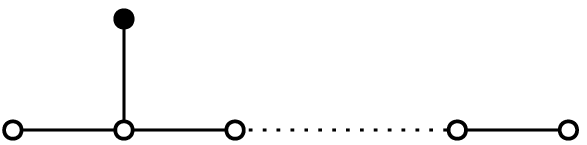}
\end{center}
For $n=6,7,8$ and $k=3$, we obtain 
$E_6$, $E_7$ and $E_8$ respectively. 
\begin{center}
\includegraphics[height=1cm]{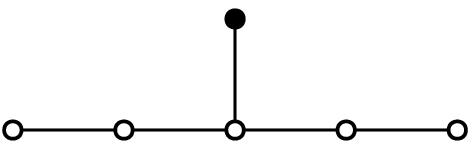}
\hskip 1cm
\includegraphics[height=1cm]{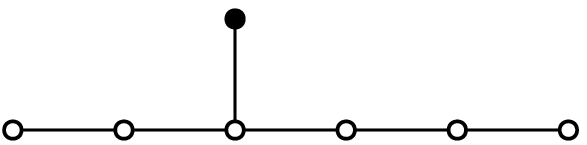}
\hskip 1cm
\includegraphics[height=1cm]{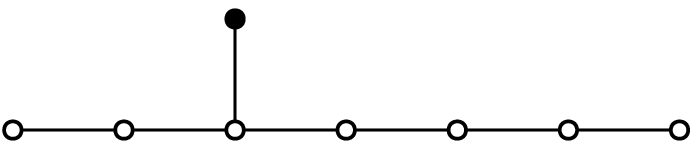}
\end{center}
The diagrams $J_{4,8}$ and $J_{3,9}$ are 
$\widetilde{E}_7$ and $\widetilde{E}_8$, respectively: 
\begin{center}
\includegraphics[height=1cm]{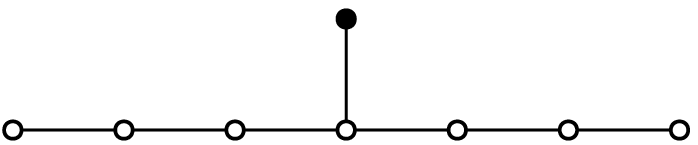}
\hskip 1cm
\includegraphics[height=1cm]{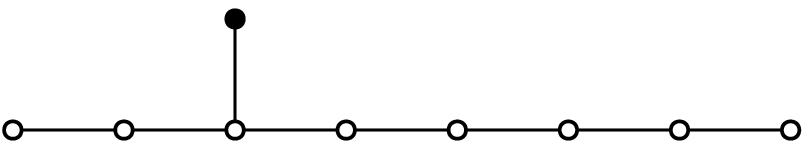}
\end{center}
There is a grading on the roots of the corresponding Kac-Moody algebra, where 
the {\em degree} is given by the coefficient of the root at $\beta$, i.e. at the $n$-th node, the 
black node in the figures.  

Zelevinsky conjectured (\cite{Zel-conj})
that the number of degree $d$ cluster variables is 
equal to $d$ times the number of real roots for $J_{k,n}$ of degree $d$. 
In the finite types, this is known to 
hold, \cite[Theorems 6,7,8]{Scott06}, whereas in the infinite cases it 
does not hold in this generality. It is expected that one needs to restrict to 
cluster variables which are associated to real roots. 
In this spirit, one can ask whether the number of 
rank $d$ rigid indecomposable modules of $\CM(\B)$ is 
$d$ times the number of real roots for $J_{k,n}$ of degree $d$. 
Jensen et al. 
confirmed that this holds in the finite type cases \cite[Observation 2.3]{JKS16}, see also 
Example~\ref{ex:fin-type} below. 

We recall a map  
from indecomposable modules in $\CM(B)$ to roots for $J_{k,n}$ via 
a map from $\ind\CM(B)$ to $\mathbb{Z}^n$ from 
\cite[Section 8]{JKS16}: If $M=L_1|L_2|\dots|L_d$ is indecomposable, let 
$\underline{a}=\underline{a}(M)=(a_1,\dots, a_n)$ be the vector where 
$a_i$ is the multiplicity of $i$ in $L_1\cup \dots \cup L_d$, for $i=1,\dots,n$. 
Let $e_1,\dots, e_n$ be the standard basis vectors for $\mathbb{Z}^n$. 
Then we can associate with $M$ a root $\varphi(M)$ for $J_{k,n}$ via the correspondence 
$\alpha_i \longleftrightarrow -e_i+e_{i+1}$, $i=1,\dots, n-1,$ and 
$\beta \longleftrightarrow e_1+e_2+\dots + e_k$. 

Note that the image of $M$ under $\underline{a}$ is in the sublattice 
$\mathbb{Z}^n(k)=\{a\in \mathbb{Z}^n\mid k\mbox{ divides }\sum a_i\}$ 
and that $\varphi(M)$ is a root 
of degree $d$. 
Via these correspondences, we 
can identify the lattice $\mathbb{Z}^n(k)$ with the root lattice of the 
Kac-Moody algebra of $J_{k,n}$ and we have the quadratic form 
$q(\underline{a})= \sum_i a_i^2 + \frac{2-k}{k^2}(\sum_i a_i)^2$ on $\mathbb{Z}^n(k)$ 
which characterizes roots for $J_{k,n}$ as the vectors with $q(\underline{a})\le 2$. 
Among them, the vectors with $q(\underline{a})=2$ correspond to real roots. 

Conjecturally, rigid indecomposable modules correspond to roots 
and if a module belongs to a homogenous tube, the associated root is imaginary. 
In the finite types, the correspondence between rigid indecomposable modules and 
(real) roots is confirmed. 
Here, we  initiate the study of 
infinite representation types, and in particular, we study rank 2 indecomposable CM-modules. We show that, in the tame cases, for every real root of $J_{k,n}$ there exist two rigid indecomposable rank 2 modules (see Theorem \ref{te:cyclemods}). Note that for $J_{4,8}$ there exist  8 rigid indecomposable rank 2 modules whose associated 
root is imaginary, see~\cite[Figure 13]{JKS16}.

\begin{ex}\label{ex:fin-type} 
For $k=2$, there are no indecomposable modules of rank 2. The diagram $J_{2,n}$ is a Dynkin 
diagram of type $D_n$ for which there are no roots of degree 2. 

Let $k=3$. \\
(i) The diagram $J_{3,6}$ is a  Dynkin diagram of type $E_6$. The only root where node 6 has degree 
$2$ is the root $\alpha_1+2\alpha_2+3\alpha_3+2\alpha_4+\alpha_5+ 2\beta$. 
It is known that there are exactly two degree $2$ cluster variables, 
cf.~\cite[Theorem 6]{Scott06}.  
On the other hand, the only rank 2 (rigid) indecomposable modules in this case are 
$L_{135}|L_{246}$ and $L_{246}|L_{135}$, cf.~\cite[Figure 10]{JKS16}. \\
(ii)
The diagram $J_{3,7}$ is a Dynkin diagram of type $E_7$. The Lie algebra of type 
$E_7$ has 7 roots of degree 2.  
There are 14 cluster variables of degree 2 (\cite[Theorem 7]{Scott06}) 
and, correspondingly, 14 rank 2 (rigid) indecomposable modules for $(3,7)$.
These can be found in~\cite[Figure 11]{JKS16}. \\
(iii)
The diagram $J_{3,8}$ is of type $E_8$, there are 28 roots of degree 2 in the corresponding root system. 
The number of cluster variables of degree 2, and the number of rank 2 (rigid) 
indecomposable modules is 56, 
cf.~\cite[Figure 12]{JKS16}. 
\end{ex}

\begin{remark}\label{rm:roots-tame}
For $J_{3,9}$ there are 84 real roots of degree 2. 
For $J_{4,8}$ there are 56 real roots of degree 2. 
One can find all these roots considering the classical result \cite[Theorem 5.6]{Kac90} 
and playing the so-called {\em find the 
highest root game} which is attributed to B. Kostant by A. Knutson \cite{KK}.
\end{remark}

%
\section{Homological properties} \label{sec:results}

The algebra $B=\B$ is Gorenstein, i.e.\ it is left and right noetherian and of finite (left and right) injective dimension. Hence, the category $\CM (B)$ is Frobenius and the projective-injective objects are the projective $B$-modules. The stable category $\uCM (B)$ has a triangulated structure in which the suspension $[ 1 ]$ coincides with the formal inverse of $\Omega$ \cite{Buc,H}.

Let $\Pi_{k,n}$ be the quotient of the preprojective algebra of type $\mathbb{A}_{n-1}$ over the ideal $\langle x^k ,y^{n-k} \rangle$. This finite dimensional $\mathbb{C}$-algebra is Gorenstein of dimension 1. The category $\CM (\Pi_{k,n})$ is equivalent to the exact subcategory $\Sub Q_k$ defined 
in~\cite{GLS08}. 
Analogously to $\CM (B)$, $\CM (\Pi_{k,n})$ is Frobenius and the stable 
category $\uCM (\Pi_{k,n})$ has a triangulated structure in which $[ 1 ]$ coincides with the formal inverse of $\Omega$, denoted by $(\Omega)^{-1}$. 
This formal inverse is not the co-syzygy $\Omega^{-1}$ since the algebras $\B$ and $\Pi_{k,n}$ are not self-injective, hence the 
slightly different notation.

By \cite[Section 4]{JKS16}, there is a (quotient) exact functor 
$\pi \colon \CM(B) \to  \CM (\Pi_{k,n})$ setting a one-to-one correspondence 
between the indecomposable modules in $\CM (\B)$ other than $P_n$ and the 
indecomposable modules in $\CM (\Pi_{k,n})$. 
This functor restricts to a triangle equivalence 
$\underline{\pi} \colon \uCM(B) \to  \uCM (\Pi_{k,n})$. 
By construction the standard triangles of $\uCM(\B)$, obtained via push-outs, are of the form
\[ A \to B \to C \to A[1],\]
where $0 \to A \to B \to C \to 0$ is a short exact sequence in $\CM (B)$. The functor $\pi$ takes exact sequences to exact sequences. This implies that we can use the additivity of the dimension vector $\dimv$ on $\CM (\Pi_{k,n})$ to reconstruct triangles, in particular, Auslander-Reiten triangles. We may refer to an Auslander-Reiten triangle 
\[ A \to B \to \tau^{-1}A  \to A[1]\]
also as the associated short exact sequence $A \hookrightarrow B \twoheadrightarrow \tau^{-1}A$.

The category $\uSub Q_k$ is triangulated and 2-Calabi-Yau \cite[Proposition 3.4]{GLS08}. Denote by $[1]_\Sub$ the shift in this category. Notice that $[1]_\Sub$ can be interpreted as the formal inverse of the syzygy when we are in $\uCM (\Pi_{k,n})$, see \cite[Remark 4.2]{JKS16}. By \cite{reiten2002noetherian}, $\tau [1]_\Sub \simeq S$, where $S$ is the Serre functor. It also holds that $S=[2]_\Sub$ from the 2-Calabi Yau condition. Therefore, $\tau = [1]_\Sub $ over $\uSub Q_k$, and this implies that $\Omega = \tau^{-1}$ in the category $\uCM (\Pi_{k,n})$. Hence, by the equivalence $\underline{\pi}$, we have $\Omega = \tau^{-1}$ in $\uCM (B)$.

\begin{figure}[h]
\[
\includegraphics[width=4cm]{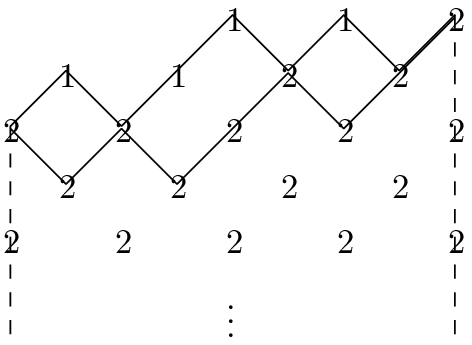}
\hspace{1.5cm}
\includegraphics[width=4cm]{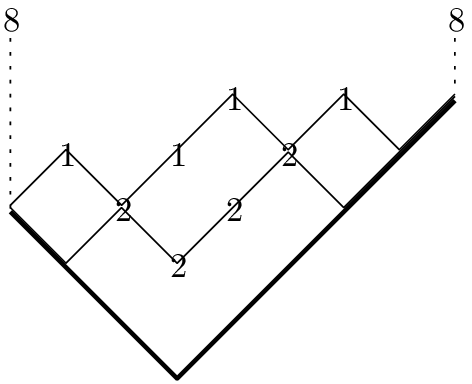}
\hspace{1.5cm}
\includegraphics[width=2.2cm]{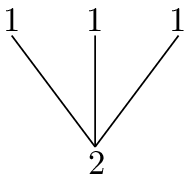}
\]
\caption{Lattice diagram of a module in $\CM(B_{3,8})$ 
and its image in $\CM(\Pi_{3,8})$ under $\pi$, and the corresponding quotient poset $(1^3,2)$}\label{fig:M-pi-of-M}
\end{figure}

\begin{ex}\label{ex:pi-rk2}
Let $(k,n)=(3,8)$. Consider the rank 2 module $M=L_I\mid L_J$ with $I=\{2,5,7\}$ and 
$J=\{1,3,6\}$. As for rank 1 modules, it is convenient to view higher rank modules as lattice 
diagrams. The lattice diagram of $M$ is drawn on the left hand side in 
Figure~\ref{fig:M-pi-of-M}. 
The image $\pi(M)$ is obtained by taking the quotient of $M$ by the projective $P_8$.
In particular, we can obtain the dimension vector of the $\Pi_{k,n}$-module $\pi(M)$ by 
cutting out the vertices corresponding to the lattice of $P_8$ from the lattice  of $L_I\mid L_J$ as in Figure \ref{fig:M-pi-of-M} (center)
and considering the multiplicities of the vertices 1 to $n-1$.   
\end{ex}

There is a covering functor from the category of finitely generated complete poset 
representations for a poset $\Gamma$ \cite[Chapter 13]{S93}, where 
$\Gamma$ is a cylindrical covering of the circular quiver of $B$, to 
the category $\CM(\B)$. 
The complete poset representations have a vector space at each vertex of $\Gamma$ 
and all arrows are subspace inclusions. 
If $\widetilde{M}$ is a complete poset representation for $\Gamma$, 
it can be identified with a finite subspace configuration of a vector space, say $M_*$, 
with $\dim M_*=\rk \widetilde{M}$. 

It is important that $\widetilde{M}$, and therefore $M \in \CM (B)$, can be identified 
with a pull-back from a finite quotient poset of $\Gamma$ 
and the 
indecomposability of $M$ can be deduced from the indecomposability of 
$\widetilde{M}$. 
Moreover when $M$ is rigid indecomposable, $\widetilde{M}$ is unique up to a grade 
shift, \cite[Lemma 6.2, Remark 6.3]{JKS16}.

\begin{remark}\label{rem:poset-indec} 
Let $M$ be an indecomposable of rank 2 in $\CM (B)$. 
Then the corresponding finite quotient poset 
has to be of the form $(1^r,2)$ for some $r$. 
Since $(1,2)$ and $(1^2,2)$ are dimension vectors for the quivers $1 \rightarrow 2$ and 
$1\rightarrow 2 \leftarrow 3$ respectively, of Dynkin type $A_2$ and $A_3$ respectively, 
the corresponding representations cannot be indecomposable. Hence $r \geq 3$. 
These posets are precisely the ones corresponding to indecomposable subspace 
configurations of rank 2.
\end{remark}

\begin{ex}\label{ex:poset}
Let $M=L_I\mid L_J$ be a rank 2 module in $\CM(\B)$ with $I\ne J$.
Then its poset is of 
the form $(1^r,2)$ for $0< r\le k$. 
The modules $M=L_I\mid L_J$ with $I=\{ 2,5,7\}$ and $J=\{1,3,6\}$ 
from Example~\ref{ex:pi-rk2} and the first module in Figure~\ref{fig:3configs} 
have poset $(1^3,2)$. 
Examples for the poset $(1,2)$ are the 
 modules in Figure~\ref{fig:2configs} and the last module in Figure~\ref{fig:3configs}. 
An example for $(1^2,2)$ is the second 
module in Figure~\ref{fig:3configs}. \\
\end{ex}


\subsection{Extension spaces between rank 1 modules} 

Let $I$ be a rim, and let $d_i$ and $l_i$ respectively be the lengths of disjoint intervals of $I$ and the lengths of the corresponding intervals of the complement of $I$ in $\{1,2,\dots,n\}$. In other words, let $d_i$ and $l_i$ denote the lengths of downward and upward slopes, respectively, of $I$. Let $m$ denote the minimum of the numbers $d_i$ and $l_i$.

\begin{prop}
Let $I$ be a rim with two peaks and let $J$ be any rim. If $I$ and $J$ are crossing, then 
$$\Ext^1(L_I, L_J)\cong \mathbb C[[t]]/ (t^a),$$ 
where $a$ is less or equal to the minimum $m$ of the lengths of the slopes (both downward and upward) of $I$. 
\end{prop}
\begin{proof}
Since there are only four slopes on the rim $I$, when $J$ is placed underneath $I$ in the computation of $\Ext^1(L_I, L_J)$, as in  \cite[Theorem 3.1]{BB}, there are at most four trapezia appearing. Since we assumed that the rims are crossing, then there are exactly four trapezia with nontrivial lateral sides, and hence, there are exactly two boxes (each box consisting of two trapezia). Each of the lateral sides of the trapezia involved is of length at most equal to the length of the corresponding slope of the rim $I$. It follows that the matrix $D^*$ is of the form 
$$
\left[\begin{array}{cc}-t^{m_1} & t^{m_2} \\t^{m_3} & -t^{m_4}\end{array}\right],
$$
where the numbers $m_i$ denote the lengths  of the lateral sides of the trapezia used to compute the extension space 
$\Ext^1(L_I, L_J)$ \cite{BB}. If we choose $a$ to be the minimal $m_i$, then the proposition follows. 
\end{proof}

\begin{cor} If $I$ and $J$ are crossing rims with two peaks each, then 
$$\Ext^1(L_I, L_J)\cong \mathbb C[[t]]/ (t^a),$$ 
where $a$ is less or equal to the minimum of the lengths of all slopes of $I$ and $J$.
\end{cor}
\begin{proof} It follows from  $\Ext^1(L_I, L_J)\cong \Ext^1(L_J, L_I)$ (see \cite[Theorem 3.7]{BB}).
\end{proof}

\begin{cor} If $I$ and $J$ are crossing rims with $I$ having two peaks and one of the slopes of length $1$, then 
$$\Ext^1(L_I, L_J)\cong  \Ext^1(L_J, L_I)\cong \mathbb C.$$ 
\end{cor}
\begin{proof} If $a$ is as in the previous proposition, then in this case  $a\le 1$, and since $I$ and $J$ are crossing, we must have $1\le a$. 
\end{proof}

We can use the previous corollary to construct part of the Auslander-Reiten quiver containing rank 1 modules whose rims have two peaks and a slope of length 1. For such a rim $I$, and  $J$ such that  $\Omega (L_I) = L_J$,  if we can find a non-trivial short exact sequence of the form $$0\rightarrow L_I \rightarrow M \rightarrow L_J\rightarrow 0,$$  
then this sequence must be an Auslander-Reiten sequence.

\begin{ex}
If $I=\{1,2,\dots,k-1\}\cup \{m\}$, where $n>m>k$, then  for any rim $J$ that is crossing with $I$ we have $\Ext^1(L_I, L_J)\cong \mathbb C$.
\end{ex}

In the following proposition, we deal with  a case where the upper bound from the previous proposition is achieved. 

\begin{prop} \label{prop:I-omegaI}
Let $I$ be a rim with two peaks, and let $m$ be as above, i.e.\ the minimum of the lengths of the slopes of $I$. Then 
$$\Ext^1(L_I,\Omega(L_{I}))\cong \mathbb C[[t]]/ (t^m).$$
\end{prop}
\begin{proof}
Since the rim $I$ has two peaks, its first syzygy is also a rank 1 module.  As in the proof of the previous proposition, from the proof of \cite[Theorem 3.1]{BB} we know that the matrix of the map $D^*$ from that proof is a $2\times2 $ matrix of the form 
$$
\left[\begin{array}{cc}-t^{m_1} & t^{m_1} \\t^{m_2} & -t^{m_2}\end{array}\right],
$$
where the numbers $m_1$ and $m_2$ denote the lengths  of the lateral sides of the trapezia used to compute the extension space 
(cf.\ \cite{BB}), as in the following picture, and that $m=\min\{m_1,m_2\}.$
\[
\begin{tikzpicture}[scale=0.6,baseline=(bb.base),
quivarrow/.style={black, -latex, thick}]
\newcommand{\seventh}{51.4} 
\newcommand{\circradius}{1.5cm}
\newcommand{\inradius}{1.2cm}
\newcommand{\outradius}{1.8cm}
\newcommand{\dotrad}{0.1cm} 
\newcommand{\bdrydotrad}{{0.8*\dotrad}} 
\path (0,0) node (bb) {}; 


\draw (-1,6) circle(\bdrydotrad) [fill=black];
\draw (-1,4) circle(\bdrydotrad) [fill=black];

\draw (0,5) circle(\bdrydotrad) [fill=black];
\draw (1,4) circle(\bdrydotrad) [fill=black];
\draw (2,3) circle(\bdrydotrad) [fill=black];
\draw (3,4) circle(\bdrydotrad) [fill=black];
\draw (4,3) circle(\bdrydotrad) [fill=black];
\draw (5,2) circle(\bdrydotrad) [fill=black];

\draw (0,3) circle(\bdrydotrad) [fill=black];
\draw (1,2) circle(\bdrydotrad) [fill=black];
\draw (2,3) circle(\bdrydotrad) [fill=black];
\draw (3,2) circle(\bdrydotrad) [fill=black];
\draw (4,1) circle(\bdrydotrad) [fill=black];
\draw (6,1) circle(\bdrydotrad) [fill=black];
\draw (7,2) circle(\bdrydotrad) [fill=black];
\draw (5,0) circle(\bdrydotrad) [fill=black];
\draw (7,0) circle(\bdrydotrad) [fill=black];

\draw[dashed] (-1,4)--(0,3)--(1,2)--node[right]{$l_2\,\,\,\,\,\,\,\,\,\,\,\,\,\,\, $}(2,3)--(3,2);
\draw[dotted] (3,2)--(4,1);
\draw[dashed](4,1)--(5,0)--node[right]{$l_1$}(6,1)--(7,0); 
\draw[-, thick, black] (5,0)--node[left]{\small$m_2\!\!$}(5,2);
\draw[-, thick, black] (1,4)--node[left]{\small${{m_1\!\!}}$}(1,2);

\draw[dotted] (3,2)--(4,1);





\draw [quivarrow,shorten <=5pt, shorten >=5pt] (-1,6) -- node[above]{$1$} (0,5);
\draw [dotted,shorten <=5pt, shorten >=5pt] (0,5)
-- node[above]{} (1,4);
\draw [quivarrow,shorten <=5pt, shorten >=5pt] (1,4) -- node[above]{$d_1$} (2,3);
\draw [quivarrow,shorten <=5pt, shorten >=5pt] (3,4) -- node[above]{} (2,3);
\draw [quivarrow,shorten <=5pt, shorten >=5pt] (3,4) -- node[above]{$\,\,\,\,\,\,\,\quad\,\,\,\,\,\,\,\,\,\,d_1+l_1+1$} (4,3);
\draw [shorten <=5pt, shorten >=5pt,dotted] (4,3) -- node[above]{} (5,2);
\draw [quivarrow,shorten <=5pt, shorten >=5pt] (7,2) -- node[above]{$n$} (6,1);
\draw [quivarrow,shorten <=5pt, shorten >=5pt] (5,2) -- node[above]{$$} (6,1);



\draw [dotted] (-1,-1) -- (-1,6);
\draw [dotted] (7,-1) -- (7,6);


\end{tikzpicture}
\]

We see from the picture that  $m_1=\min\{d_1,l_2\}$ and $m_2=\min\{d_2,l_1\}$. The proposition now follows.

\end{proof}

\begin{cor}\label{cor:ext1}
Let $I$ be a rim with two peaks and one of the slopes (either downward or upward) of length $1$. Then 
$$\Ext^1(L_I,\Omega(L_{I}))\cong \Ext^1(\Omega(L_{I}), L_I)\cong \mathbb C.$$
\end{cor}

%
\subsection{Auslander-Reiten sequences.}

The purpose of this subsection is to determine certain Auslander-Reiten sequences of the form 
$L_I \to M \to L_J$,  where $L_I$ and $L_J$ are rank 1 $\B$-modules. To do this we move back and forth from $\CM (\B)$ to $\CM (\Pi_{k,n})$ using the quotient functor $\pi$.

\begin{remark}\label{Remark-rigidM} Let $L_I$ and $L_J$ be two rank 1 modules such that $\dim \Ext^1 (L_J, L_I) = 1$. Using the quotient functor, the modules $\pi(L_I)$ and $\pi(L_J)$ are rigid modules over $\CM (\Pi_{k,n})$ (or one may consider them as modules in the subcategory $\Sub Q_k $ of the preprojective algebra). Then, by \cite[Proposition 5.7]{rigid} the middle term $\pi(M)$ of the 
non-trivial extension is rigid. 
Thus, the middle term $M$ is rigid.
\end{remark}

\begin{defi}
Let $I$ be a $k$-subset of $\mathbb{Z}_n$ consisting of two intervals where one of the intervals 
is a single element. We call such a $k$-subset {\em almost consecutive}. 
\end{defi} 

From \cite[Section 2]{BB}, we know that if $I$ is almost consecutive, say $I = \{i, j, \ldots, j+k-2\}$ 
for some $j\in [i+2,\dots, n-k+i+1]$, then 
$\Omega (L_I)=L_J$,  where $J= \{i+1, \ldots, i+k-1, j+k-1\}$ is also almost consecutive. 

%
%
%

\begin{theorem}\label{Teo rk1-new} 
Let $I = \{i, j, \ldots, j+k-2\}$ be almost consecutive 
and $J$ be such that $L_J=\Omega (L_I)$. 
Then the Auslander--Reiten sequence with $L_I$ and $L_J$ as end-term is as follows: \\
\[ L_I \hookrightarrow \frac{L_X}{L_Y}  \twoheadrightarrow L_J,\] 
with $X= \{i+1, j,j+1, \ldots, j+k-3, j+k-1\}$ and $Y= (I\cup J)\setminus X$ 
and where $\frac{L_X}{L_Y}$ is indecomposable if and only if 
$j\ne i+2$. 
In case $j=i+2$, $\frac{L_X}{L_Y}=P_i\oplus L_U$ for $U=\{i,i+2,i+3,\dots,k+i-1,k+i+1\}$.

Furthermore, in both cases, the middle term is rigid. 
\end{theorem} 
\begin{proof}
We will prove the claims for $i=1$, the statement then follows from the symmetry of $B$. By Corollary \ref{cor:ext1}, we have $\dim \Ext^1 (L_J,L_I) =1$, so the middle term $M$ is rigid by Remark \ref{Remark-rigidM}. Note that $M$ is a rank 2 module, so it is either a direct sum of two rank 1 modules or indecomposable module of rank 2. 
The projective cover of $M$ is a direct summand of the direct sum of the projective covers of 
$L_I$ and $L_J$, i.e.\  a summand  of $P_0\oplus P_{j-1} \oplus P_1\oplus P_{j+k-2}$. 

Suppose that $M=L_U\oplus L_V$. By the above, the peaks of $M$ belong to 
$\{0,1,j-1,j+k-2\}$. 

(i) If $M$ has a projective summand, say $L_V=P_a$ for some $a$, we have an irreducible 
monomorphism $L_I\hookrightarrow P_a$, hence $L_I=\rad(P_a)$. 
In that case, $I=\{a,a+2,\dots, a+k\}$, i.e.\  $a=1$, $j=3$ 
and $L_U$ is as claimed in part (2) of the theorem. 

(ii) If none of the summands of $M$ are projective, they have two peaks each, so 
$U=A\cup B$ and $V=C\cup D$ are two-interval subsets of $C_1$. 

We first claim that if the vertices $1$ and $j+k-2$ are the two peaks of $L_U$, 
then $L_U\cong L_J$. 
To 
see this, 
let $U=A\cup B=\{2,\dots\} \cup \{j+k-1,\dots \}$. 
We use that $\pi$ takes exact sequences to exact sequences. One checks that 
the dimension of $\pi(L_I\oplus L_J)$ is zero at 
vertices $j+k-1$ and  $j+k,\dots, n$ (see the following two figures). 
$$ 
\includegraphics[width=4cm]{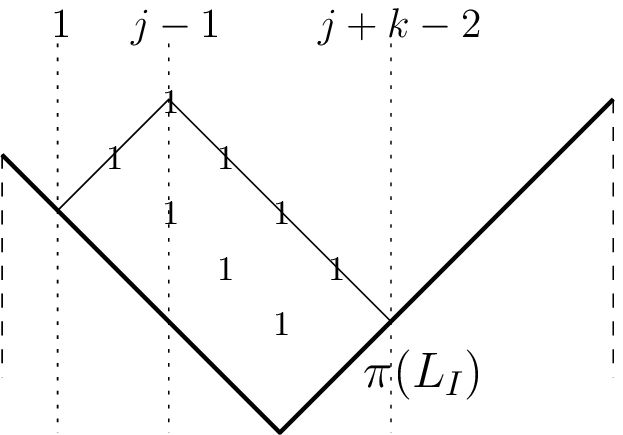}
\hskip 1cm
\includegraphics[width=4cm]{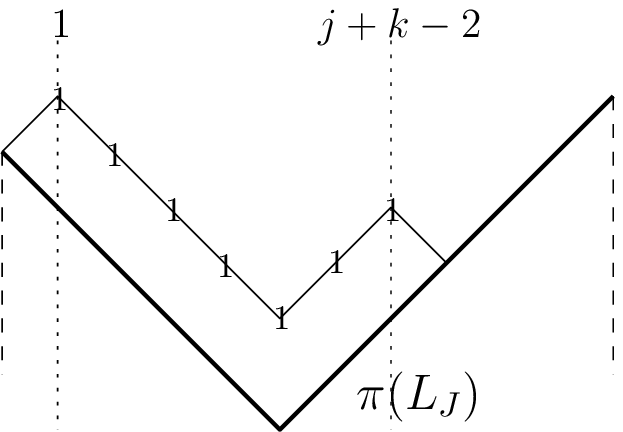}
$$
If $|B|>1$, then $\pi(L_U)$ has positive dimension at vertex $j+k-1$  (see the following figure). 
$$ 
\includegraphics[width=4cm]{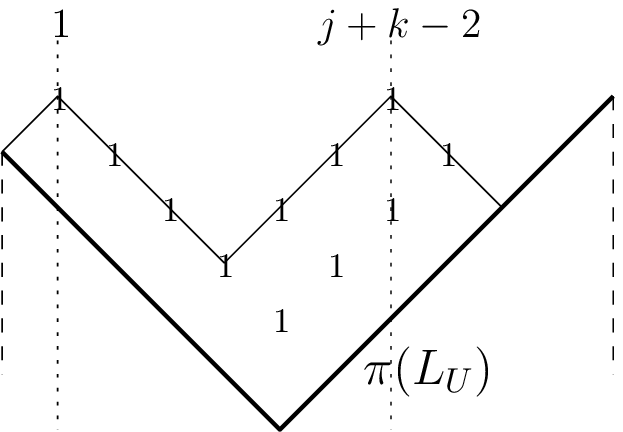} 
$$
Hence $|B|=1$, i.e.\ $U=J$. 

Now we claim that the exact sequence 
$L_I\to L_V\oplus L_J \to L_J$ is not an Auslander--Reiten sequence. 
If $L_J$ is a summand of a middle term of such Auslander--Reiten sequence, 
one can complete $L_J$ to a cluster tilting object $T$ as in Remark~\ref{rem:cto-exists}. 
Then the quiver of the endomorphism algebra $\End_{\uCM(B)}(T)$ 
has a loop at the vertex $L_J$ 
corresponding to the irreducible map $L_J\to L_J$, which is a contradiction.  

So $U$ and $V$ both contain one peak of $I$ and one of $J$. Say $U$ contains 
the peak at $1$ and $V$ the peak at $j+k-2$. Since $0$ cannot be a peak of $U$, we have 
$U=\{2,\dots\} \cup \{j,\dots\}$ and $V=C\cup D= \{1,\dots\}\cup \{j+k-1,\dots\}$. 
By the same argument as above, $|D|=1$. 
Applying $\pi$ to our short exact sequence yields that the dimension at vertex 2 of 
$\pi(L_U)$ is 1, whereas at vertex 2 of  $\pi(L_V)$ it is $0$. However, $\pi(L_I\oplus L_J)$ has 
dimension 2 at  vertex 2, which is  a contradiction. 

We assume now that $M$ is indecomposable. By the discussion above (Remark \ref{Remark-rigidM}) $M$ is rigid so it is determined by its profile, and the profile provides a filtration by rank 1 modules, $M=L_X|L_Y$.  
There are infinitely many injective maps of $L_I$ into $M$. They differ by the relative positions of 
$L_I$ and the quotient $L_J$ on the lattice diagram of $M$ 
covering the dimension vector of 
$M$. There are five distinct cases: the rim of $L_I$ can be strictly lower than the rim of $L_J$, 
they can touch, intersect properly, or the rim of $L_J$ can be below $L_I$, touching 
or being strictly lower, as in Figures~\ref{fig:2configs} and~\ref{fig:3configs}. 

\begin{figure}[h]
\[
\includegraphics[width=5cm]{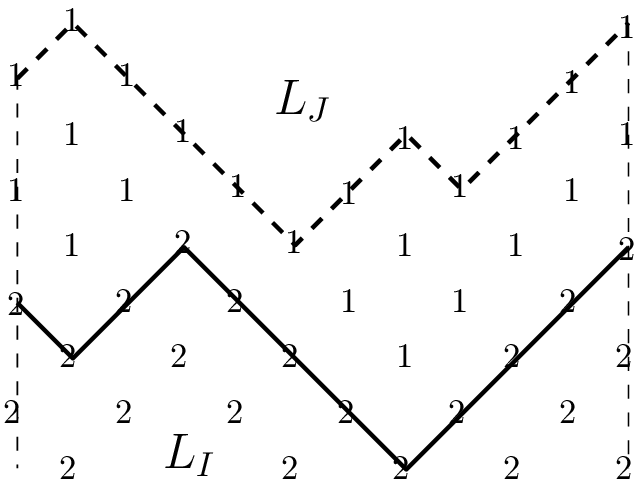}
\hspace{1cm}
\includegraphics[width=5cm]{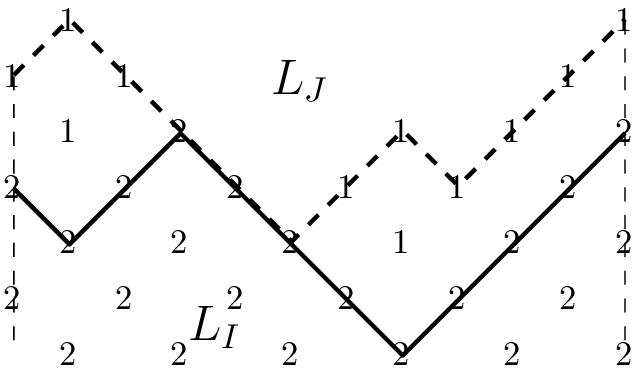}
\]
\caption{Rank 2 modules with submodule $L_I$ and quotient $L_J$}\label{fig:2configs}
\end{figure}
\begin{figure}[h]
\[
\includegraphics[width=4.8cm]{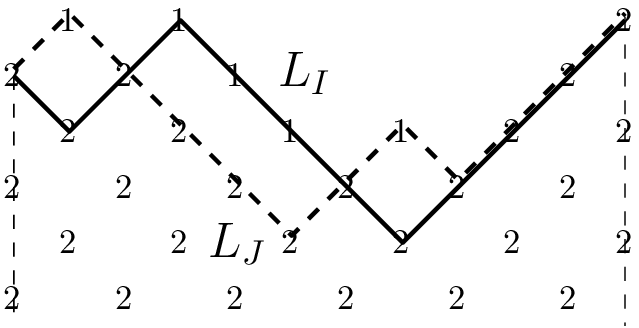}
\hspace{.5cm}
\includegraphics[width=4.8cm]{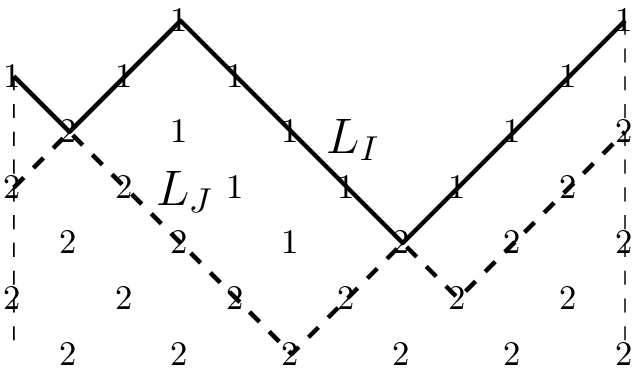}
\hspace{.5cm}
\includegraphics[width=4.8cm]{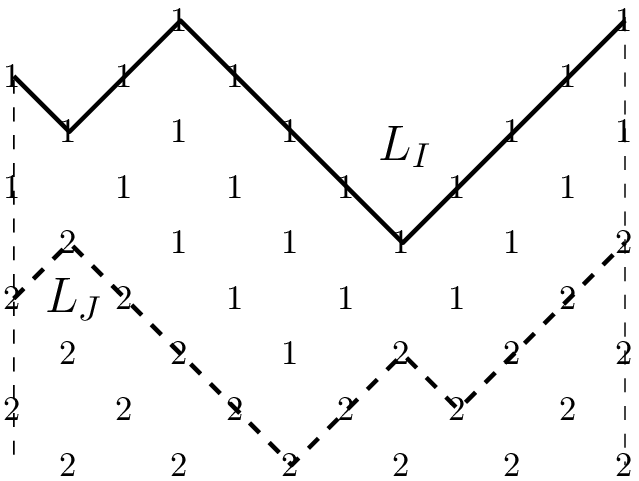}
\]
\caption{Rank 2 modules with submodule $L_I$ and quotient $L_J$, continued}\label{fig:3configs}
\end{figure}
%
Note that all these representations have up to three disjoint regions 
with 1-dimensional vector spaces at finitely many vertices 
mapping to the infinite region with 2-dimensional vector spaces at every vertex. The only case providing a profile of an indecomposable module is the left one in Figure \ref{fig:3configs}, yielding  $X=\{  2, j,\dots,j+k-3, j+k-1   \}$   and $Y=\{ 2,4,\dots,k, j+k-2    \}$
\end{proof} 

Note that in the previous theorem, $X=I\cup\{i+1,j+k-1\}\setminus\{i,j+k-2\}$, whereas $Y=J\cup\{i,i+k-2\}\setminus \{i+1,j+k-1\}$.

\begin{remark}\label{rem:rim-sequences}
In case $k=3$, Theorem~\ref{Teo rk1-new} covers all possible Auslander-Reiten sequences where the 
start and end terms are of rank 1. 
By \cite[Section 2]{BB}, the rank of $\tau^{-1}(L_J)$ is 
one less than the number of peaks of $J$, so $\rk\tau^{-1}(L_J)=1$ if and only if 
$J$ is a 2-interval subset. Since $k=3$, the 2-interval subsets are of the form $\{i,i+1,i+m\}$ with $2<m<n-1$. 
\end{remark}

Assume $k=3$. If we are given a rank 2 module $M$ with profile $X|Y$, then how do we know if it is the middle term of an Auslander--Reiten sequence with rank 1 modules $L_I$ and $\Omega(L_I)$? From the previous theorem, and the diagram on the left hand side in the previous picture, in order  for $M$ to be such a module, $X$ and $Y$ must be $3$-interlacing (see the first definition in Section \ref{sec:rk2-roots-comb}), and when drawn one above the other (as in the above picture), the diagram we obtain has to contain three consecutive diamonds, with no gaps between them and with the end two diamonds of lateral size 1. In order to recognize the rims of $L_I$ and $\Omega(L_I)$ from the profile $X|Y$, the easiest thing to do is to identify the middle diamond if it happens to be of size greater than 1 (as in the above picture). If the middle diamond is of size 1, then we can identify the right hand side diamond, since it is followed by an upward "tail like'' portion of the rim covered by both rim $X$ and rim $Y$  in the above picture. If there is no tail, then we only have three diamonds of size 1, and then we deal with the module $\{1,3,5\}|\{2,4,6\}$. 

Other rigid indecomposable rank 2 modules whose profile $X|Y$ does not satisfy these conditions do not appear as the middle term of an Auslander--Reiten sequence with rank 1 modules. 
In the tame cases, they either appear at the mouth of a tube, or they are meshes of the modules with at least one of them of rank greater than 1 as we will see later. 


\subsection{Periodicity} \label{sec:periodic}

It follows from the direct computation in \cite[Proposition 2.7]{BB}, that if $L_I$ is a rank $1$ module, then $\Omega^2 (L_I)= L_{I+k}$, where $(I_1 \mid \ldots | I_n) +m$ is the profile $(I_1+m) | (I_2+m) | \dots | (I_n+m)$ 
obtained by adding $m$ to each number in every $k$-subset appearing in the profile. Set $v=\lcm(n,k)/k$. This leads to the next observation.

\begin{remark}\label{prop:rank 1 regular}
Let $L_I$ be a non-projective rank $1$ module, and let $M$ be a module in the $\tau$-orbit of $L_I$. 
Then, $M$ is $\tau$-periodic of period $d$, for some factor $d$ of $2v$.
\end{remark}
It is not difficult to show that the same formula $\Omega^2(L_I|L_J)=L_I+k|L_J+k$  holds for the rigid rank 2 modules. 

By graded Morita theory, $\CM(B)$ is equivalent to $\CM^{\mathbb{Z}_n}(R_{k,n})$ where 
$R_{k,n}=\mathbb{C}[x,y]/(x^k - y^{n-k})$ and where the $\mathbb{Z}_n$-grading is given by 
$\deg x=1$ and $\deg y=-1$, \cite[Theorem 3.16]{DL16}. For the latter, 
Demonet-Luo show, \cite[Theorem 3.22]{DL16}, that there is an isomorphism of autoequivalences $[2] \simeq (-k)$, 
where the notation $(-1)$ refers to a shift in the $\mathbb{Z}_n$-grading. From that, we obtain:


\begin{prop}\label{prop:keller}
Every module in $\uCM(\B)$ is $\tau$-periodic with period a factor of $2v$. 
\end{prop}


\subsection{A different approach on periodicity: Products of Dynkin types}\label{general case}

We now switch perspectives and use \cite{BKM16} to follow an approach of Keller, \cite{Keller2013}. 
This will allow us to identify the tame cases later. 

The category $\uCM(\B)$ has cluster-tilting objects whose 
endomorphism algebras have rectangular quivers built by 
$(k-1)\times (n-k-1)$ lines of arrows, forming alternatingly oriented squares.
These are exactly the quivers of the rectangular 
arrangements from~\cite[Section 4]{Scott06}. We will denote them by $Q \square Q'$, where $Q$ 
is a Dynkin quiver of type $A_{k-1}$ and $Q'$ of type $A_{n-k-1}$, with corresponding Coxeter numbers $h=k$ and $h'=n-k$, as in~\cite{Keller2013}. 
These are quivers with a natural potential $P$, in the sense of 
\cite{DWZ08}, given by the 
sum of all clockwise cycles minus the sum of all anti-clockwise cycles. We will write 
QP to abbreviate `quiver with potential'.

\begin{ex}\label{ex:tame}
For $(3,9)$ and for $(4,8)$ the rectangles $Q \square Q'$ are 
as follows: 
\[
\xymatrix@R=0.4em@C=0.6em{ 
&& && && &&  && && 
\bullet\ar[rr] && \bullet\ar[dd] && \bullet\ar[ll] \\ 
\bullet \ar[rr] && \bullet\ar[dd] && \bullet\ar[rr]\ar[ll] && \bullet\ar[dd] && \bullet\ar[ll]   \\
 && && && && 
 && && 
\bullet\ar[uu]\ar[dd] && \bullet\ar[rr]\ar[ll] && \bullet\ar[uu]\ar[dd]  \\ 
\bullet\ar[uu] && \bullet\ar[ll]\ar[rr] && \bullet\ar[uu] && \bullet \ar[ll]\ar[rr] && \bullet \ar[uu]  \\
 &&&&&& &&&&&& \bullet\ar[rr] && \bullet\ar[uu] && \bullet\ar[ll] 
}
\]
\end{ex}

Such Jacobian algebras can also be obtained as 2-Calabi-Yau tilted algebras from certain (Hom-finite) generalized cluster categories $\Cc_A$ 
in the sense of \cite{Amiot2009}, since the corresponding QP are (QP-)mutation equivalent to $(Q \boxtimes Q',\widetilde{P})$, following \cite{Keller2013}.

In the cases $(3,9)$ and $(3,8)$, these QP are 
\[
\xymatrix@R=0.4em@C=0.6em{ 
&& && && &&  && && 
\bullet\ar[rr] && \bullet\ar[lldd]\ar[rr] && \bullet\ar[lldd]  \\ 
\bullet \ar[rr] && \bullet\ar[lldd]\ar[rr] && \bullet\ar[rr]\ar[lldd]  && \bullet\ar[rr]\ar[lldd]  
 && \bullet\ar[lldd]   \\
 && && && && 
 && && 
\bullet\ar[uu]\ar[rr] &&  \bullet\ar[lldd]\ar[rr]\ar[uu] && \bullet\ar[lldd]\ar[uu]  \\ 
\bullet\ar[uu]\ar[rr] && \bullet\ar[uu]\ar[rr] && \bullet\ar[uu]\ar[rr] && \bullet \ar[uu]\ar[rr] 
 && \bullet \ar[uu]  \\
 &&&&&& &&&&&& \bullet\ar[rr]\ar[uu] && \bullet\ar[uu]\ar[rr] && \bullet \ar[uu]
}
\]
with potential the sum of all positive 3-cycles minus the sum of all negative 3-cycles (note that this agrees with the natural potential of the dimer model from~\cite[Section 3]{BKM16}). By \cite[Section 2.1]{KellerReiten07}, there exists an equivalence of categories 
\begin{equation*}
\mathcal{C}_A/(\Sigma T) \leftrightarrow \mmod Jac \hspace{1pt} (Q,P) \leftrightarrow \uCM (\B) /((\Omega)^{-1} T'),
\end{equation*}
where $(\Sigma T)$ is the ideal of morphisms that factor through 
$\mathrm{add} \hspace{1pt} \Sigma T$ for a cluster tilting object $T$ (respect. $T'$ and $(\Omega)^{-1} T'$) over the 2-CY category.
These categories are Krull-Schmidt, so by \cite[Section 3.5]{KellerReiten07}, 
the Auslander--Reiten quiver of $\mmod Jac \hspace{1pt} (Q,P)$ is obtained from the 
Auslander--Reiten quiver of the 2-CY categories removing a finite number of vertices.

\begin{remark} If we denote $\Sigma$ by the shift functor, $S$ the Serre functor, and $\tau_\Cc$ the Auslander--Reiten translation over $\Cc_A$, there is an isomorphism of functors $\Sigma^2= S$. On the other hand $S$ is $\Sigma \tau$ over $\Cc_A$, so $\tau_{\Cc} = \Sigma$. Keller's proof of the periodicity of the Zamolodchikov transformation $(\tau \otimes 1)$ \cite[Theorem 8.3]{Keller2013} indicates a way to prove $\tau$-periodicity. In fact, one can show that $\tau_\Cc$ is $2n$-periodic. 

Keller shows that $(\tau \otimes 1)^h = \Sigma^{-2}$ is an isomorphism of functors of $\Cc_A$ and with this that $(\tau \otimes 1)^{h+h'} = \mathds{1}$ is isomorphism of functors of $\Cc_A$. 
We can do the following:
\[\mathds{1}= \mathds{1}^h=(\tau \otimes 1)^{(h+h')h} = (\tau \otimes 1)^{h^2} (\tau \otimes 1)^{h h'} = (\Sigma^{-2})^h (\Sigma^{-2})^{h'} = \Sigma^{-2h - 2h'} = \tau_\Cc^{-2(h + h')}\]

In our case $2(h+h')=2n$. 
\end{remark}

%

\subsection{Tame cases}\label{sec:tame-rep-theory}

For $k=2$, such a quiver $Q\square Q'$ is of type $A_{n-3}$, 
for $(3,6)$ of (mutation) type $D_4$, for $(3,7)$ of type $E_6$ and for $(3,8)$ of type $E_8$. 
The Dynkin types above 
are the only cases of finite representation type, whereas the two cases $(3,9)$ and 
$(4,8)$ in 
Example~\ref{ex:tame} are the first cases of infinite representation type. The quivers with 
potential give rise to path algebras with relations whose mutation class representation theory is studied in \cite{GG15}. They correspond to \emph{elliptic types} $E_{7,8}^{(1,1)}$, known to be tame \cite[Theorem 9.1]{GLab}. Besides the cases $(3,9), (4,8)$ and the finite types, all other QP algebras obtained from rectangular arrangements are of \emph{wild} type.

In the cases $(3,9)$ and $(4,8)$ the QP's also arise from two cases of 2-Calabi Yau categories 
$\mathcal{C}(A)$, called \emph{tubular cluster categories} of types $(6,3,2)$ and $(4,4,2)$ respectively. It is known that 
$\mathcal{C}(A)$ is formed by a coproduct of tubular components $\coprod_{x\in \mathbb{X}} \mathcal{T}_x$ 
where almost all tubes $\mathcal{T}_x$ are of rank 1, and there are finitely many tubes of ranks $6,3,2$ (resp. $4,4,2$) \cite{Barot}.
Therefore, the corresponding Auslander--Reiten quivers are formed by finitely many tubes of ranks $6,3,2$ (resp. $4,4,2$), and infinitely many homogeneous tubes.

%
\section{Higher rank modules}\label{sec:higher-rank}

In this section, we give a construction for higher rank modules in $\CM(\B)$ in the spirit of the 
definition of rank 1 modules (Definition~\ref{d:moduleMI}). 
Recall that the rank 1 modules in $\CM(\B)$ are in bijection with $k$-subsets of $\mathbb{Z}_n$. 
Modules of higher rank can also be described combinatorially, 
as we will see here. 
%
%
%
%
For this, let $\mathrm{Id}_s$ be the $s\times s$-identity matrix, let $E_{i,j}$ be the 
elementary $s\times s$-matrix with entry 1 at position $(i,j)$ and 0 everywhere else. 
We set $\sigma$ to be the following $s\times s$-matrix: 
\[
\sigma= \sum_{i=1}^{s-1}E_{i,i+1} + tE_{s,1}=
\begin{pmatrix} 
	0 & 1 & 0 && & \\
	0 & 0 & 1 && \\
	\vdots & & \ddots &  &\\ 
	0 & & &  0 & 1 \\
	t & 0 && & 0
\end{pmatrix} 
\]

\begin{lemma}\label{lm:tau-powers}
For $j=0,\dots, s$ we have 
\[
\sigma^j=\sum_{i=1}^{s-j}E_{i,i+j} + t\sum_{i=1}^j E_{s+i-j,i}
.\]
In particular, $\sigma^s = t \mathrm{Id}_s$.\end{lemma}
\begin{proof}
Straightforward. 
\end{proof}

Let $I_1,\dots, I_s$ be $k$-subsets. We then construct a rank $s$ $\B$-module 
$L(I_1,\dots,I_s)$. The maps $x_i$ and $y_i$ depend on the number of $k$-subsets 
among $I_1,\dots, I_s$ the index $i$ belongs to.
For $i\in[1,n]$ set $r_i=|\{I_j\mid i\in I_j\}|$. 

\begin{defi}\label{def:construction} 
Let $I_1,\dots, I_s$ be $k$-subsets of $[1,\dots,n]$. We define 
$L(I_1,\dots, I_s)$ as follows. \\
For $i=1,\dots, n$ let $V_i:=e_i L(I_1,\dots,I_s)=\C[[t]]\oplus \C[[t]]\oplus\dots\oplus\C[[t]]=Z^s$. 
Depending on the number of $k$-subsets $i$ belongs to, we define maps  
$x_i, y_i:Z^s\to Z^s$ as follows: 
%
\[
\begin{array}{ccl} 
	x_i:V_{i-1}\to V_i & \mbox{multiplication by} & \sigma^{s-r_i}  \\ 
	y_i:V_i\to V_{i-1} & \mbox{multiplication by} & \sigma^{r_i} 
\end{array}
\]
\end{defi}

\begin{remark}\label{rem:xy-ok}
Note that 
we have 
$x_iy_i=y_{i+1}x_{i+1}=\sigma^s=t \mathrm{Id}_s$ for all $i$ (Lemma~\ref{lm:tau-powers}). 
To see that $L(I_1,\dots, I_s)$ is a $\B$-module, we need to check that the 
relations $x^k=y^{n-k}$ hold. 
\end{remark}

\begin{remark}
The module $L(I_1,\dots,I_s)$ can be represented by a lattice diagram
$\mathcal{L}_{I_1,\dots,I_s}$ obtained by overlaying the lattice diagrams $\mathcal L_{I_j}$ 
such that $\mathcal L_{I_j}$ is above 
$\mathcal L_{I_{j+1}}$ for $j=1,\dots, s-1$, where the rims of $\mathcal L_{I_j}$ and 
of $\mathcal L_{I_{j+1}}$ are meeting at at 
least one vertex and possibly share 
arrows but have no two-dimensional intersections. 
The vector spaces $V_0,V_1,V_2,\ldots, V_n$ are represented by columns from
left to right (with $V_0$ and $V_n$ to be identified).
The vertices in each column correspond to some monomial 
$\mathbb C$-basis of $\mathbb C[t]\oplus\C[t]\oplus\cdots\oplus\C[t]$, 
depending on the sets $I_1,\dots, I_s$. 
Note that the $k$-subset $I_1$ can then be read off as
the set of labels on the arrows pointing down to the right which are exposed
to the top of the diagram. Labels of successive $k$-subsets can be read off from successive 
levels in the diagrams. 
To illustrate this for $s=2$, the lattice picture $\mathcal{L}_{I,J}$ for 
$I=\{2,5,8,9\}$ and $J=\{1,3,7,8\}$, $(k,n)=(4,9)$, is shown in 
Example~\ref{ex:construction}. 
\end{remark}

\begin{ex}\label{ex:construction}
Let $(k,n)=(4,9)$. Consider the $4$-subsets 
$I=\{2,5,8,9\}$ and $J=\{1,3,7,8\}$.  
The module $L(I,J)$ is illustrated in Figure~\ref{fig:lattice-embedding}. 
\begin{figure}[h]
\[
\includegraphics[scale=.6]{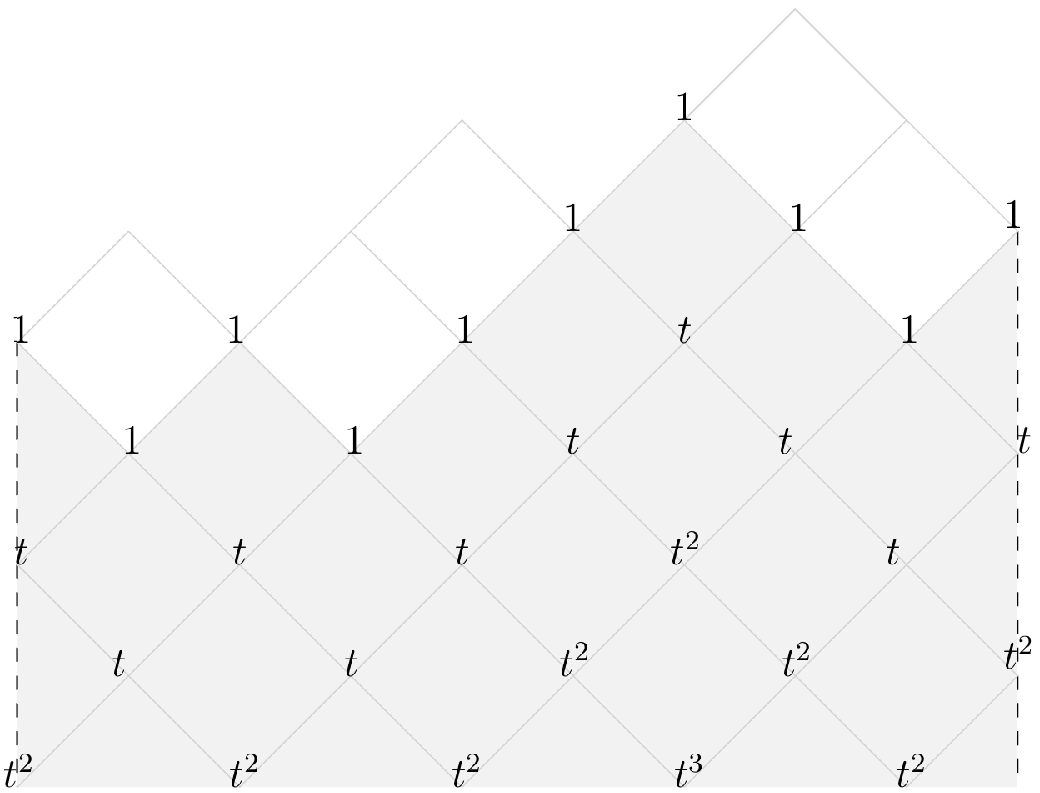}
\hskip 1cm
\includegraphics[scale=.6]{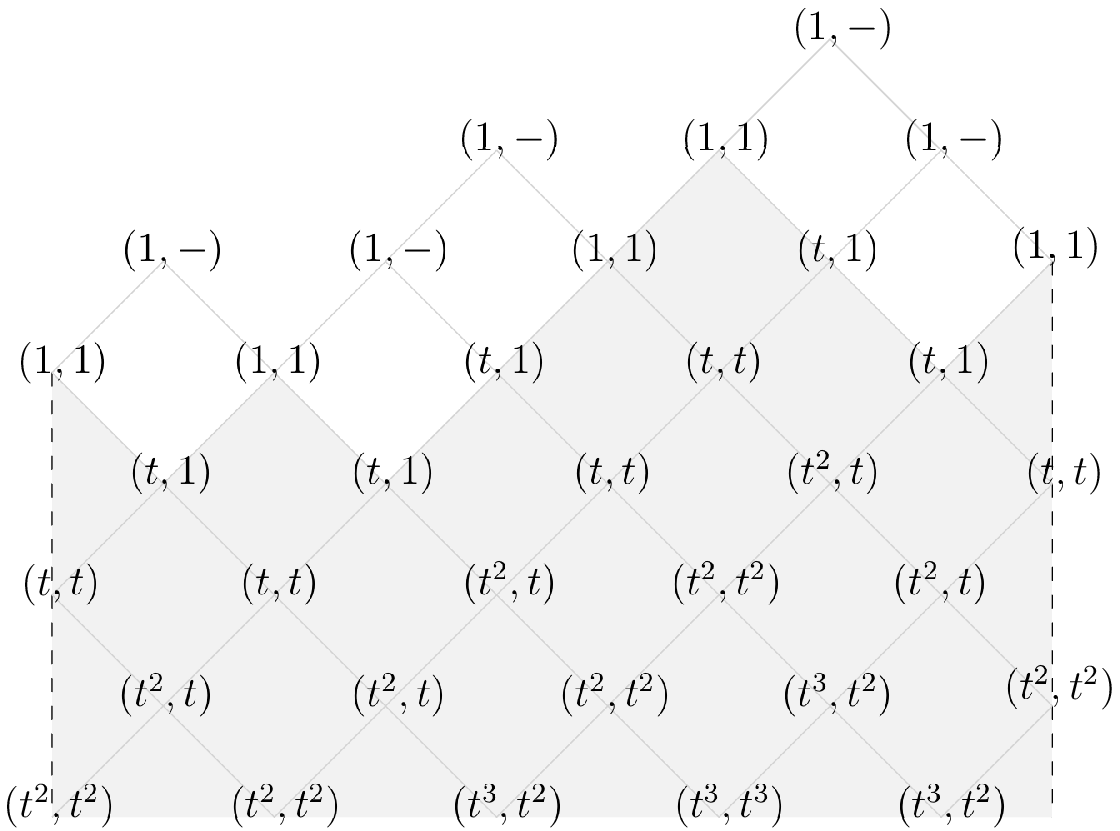}
\]
\caption{Lattice diagrams for $L(\{1,3,7,8\})$ and $L(\{2,5,7,8\},\{1,3,7,8\})$, 
$n=9.$}\label{fig:lattice-embedding}
\end{figure}
\end{ex}

\begin{ex}
(1) Let $s=2$. Then 
$\sigma=\begin{pmatrix} 0 & 1  \\ t & 0 \end{pmatrix}$ and 
$\sigma^2=\begin{pmatrix} t & 0  \\ 0 & t \end{pmatrix}$. 
Here, $x_i=y_i=\sigma$ for $i\in (I_1\setminus I_2)\cup(I_2\setminus I_1)$, 
$x_i=I_2$, $y_i=t\cdot I_2$ for $i\in I_1\cap I_2$ and $x_i=t I_2$, $y_i=I_2$ for 
$i\in I_1^c\cap I_2^c$. \\

\noindent
(2) Let $s=3$. Then 
$\sigma=\begin{pmatrix} 0 & 1 & 0 \\ 0 & 0 & 1 \\ t & 0 & 0 \end{pmatrix}$, 
$\sigma^2=\begin{pmatrix} 0 & 0 & 1 \\  t & 0 & 0\\ 0 & t & 0 \end{pmatrix}$ 
and 
$\sigma^3=\begin{pmatrix} t & 0 & 0 \\ 0 & t & 0 \\ 0 & 0 & t \end{pmatrix}$. 
An instance of this is in~\cite[Example 6.5]{JKS16} where 
$I_1=\{3,6,8\}$, $I_2=\{2,5,8\}$ and $I_3=\{1,4,7\}$ for $(k,n)=(3,8)$. 
\end{ex}

\begin{prop}\label{prop:higher-rank}
$L(I_1,\dots, I_s)\in \CM(\B)$. 
\end{prop}

\begin{proof}
We first check that $L(I_1,\dots, I_s)$ is a $\B$-module. By Remark~\ref{rem:xy-ok}, 
it remains to 
show that the relations $x^k=y^{n-k}$ hold. By construction, all the $x_i$ and all the $y_i$ 
commute. 
We have 
\[
\begin{array}{lcl}
  x^k & = & x_{i+k-1} x_{i+k-2}\cdots x_{i+1}x_i= \sigma^{s-r_{i+k-1}} 
 \cdots\sigma^{s-r_{i+1}}\sigma^{s-r_{i}} = \sigma^{sk-(r_{i+k-1} + r_{i+k-2} + r_{i+1} + r_i)}\\
  y^{n-k} & = & y_{i+k} y_{i+k+1}\cdots y_{i-2}y_{i-1}=\sigma^{r_{i+k}}\sigma^{r_{i+k+1}} 
 \cdots \sigma^{r_{i-2}}\sigma^{r_{i-1}} = \sigma^{r_{i+k} + r_{i+k+1} + \dots + r_{i-2} + r_{i-1}}
\end{array}
\]
Now $r_1+r_2+\dots +r_{n-1} + r_n=sk=|\bigcup_{j=1}^s I_j|$ and so: 
\[
\begin{array}{lcl}
 \sigma^{r_{i+k-1} + r_{i+k-2} + r_{i+1} + r_i} x^k & = &  \sigma^{sk} \\
 \sigma^{r_{i+k-1} + r_{i+k-2} + r_{i+1} + r_i}y^{n-k} & = 
  & \sigma^{r_{1}+r_2 + \dots + r_{n-1} + r_n} = \sigma^{sk}
\end{array}
\]
Since all the powers of $\sigma$ are invertible, we get $x^k=y^{n-k}$. 

Consider $L(I_1,\dots, I_s)$ as a $Z$-module. It is a direct sum 
$V_1\oplus V_2\oplus \dots\oplus V_n$ where each of the $V_i$ has a basis of size $s$, 
hence $L(I_1,\dots, I_s)$ is free over $Z$. 
\end{proof}

%
\section{Rank $2$ modules and root combinatorics} \label{sec:rk2-roots-comb}
%

%


In this section, we deal with rigid indecomposable rank 2 modules and relate them with roots for 
associated Kac--Moody algebras.

\begin{defi}[$r$-interlacing]
Let $I$ and $J$ be two $k$-subsets of $[1,n]$. 
%
$I$ and $J$ are said to be {\em $r$-interlacing} if there exist subsets 
$\{i_1,i_3,\dots,i_{2r-1}\}\subset I\setminus J$ and $\{i_2,i_4,\dots, i_{2r}\}\subset J\setminus I$ 
such that $i_1<i_2<i_3<\dots <i_{2r}<i_1$ (cyclically) 
and if there exist no larger subsets of $I$ and of $J$ with this property. 
\end{defi}

If $I$ and $J$ are $r$-interlacing, then the poset of $I\mid J$ is $(1^r,2)$, see Figure~\ref{ex:pi-rk2} for $r=3$. The module in question is indecomposable for 
$r\ge 3$ (see Remark~\ref{rem:poset-indec}).

\begin{prop}
Let $I$ and $J$ be $r$-interlacing. Then there exist $0\leq a_1\leq a_2\leq \dots a_{r-1}$ such that, as $Z$-modules, 
$${\rm Ext}^1(L_I,L_J)\cong \mathbb C[[t]]/(t^{a_1})\times \mathbb C[[t]]/(t^{a_2}) \times \cdots \times \mathbb C[[t]]/(t^{a_{r-1}}). $$
\end{prop}
\begin{proof}
We assume that we have drawn the rims $I$ and $J$ one above the other, say $I$ above $J$, as in the proof of Theorem 3.1  in \cite[Section 3]{BB}. For every $i_{2s}\in J\setminus I$ we have that rims $I$ and $J$ are not parallel between points $i_{2s-1}$ and $i_{2s}$, yielding a left trapezium.  Similarly, for every $i_{2s+1}\in I\setminus J$ we have that rims $I$ and $J$ are not parallel between points $i_{2s}$ and $i_{2s+1}$, yielding a right trapezium. Since $I$ and $J$ are $r$-interlacing, we have, in alternating order $r$-left and $r$-right trapezia, giving us in total $r$ boxes. The statement now follows from the proof of Theorem 3.1 in \cite{BB} which says that ${\rm Ext}^1(L_I,L_J)$ is a product of $r-1$ cyclic $Z$-modules. 
\end{proof}

\begin{cor}\label{cor:dim-ext}
Let $k=3$ and $I$ and $J$ be rims. Then ${\rm Ext}^1(L_I,L_J)\cong \mathbb C \times \mathbb C$  if and only if $I$ and $J$ are $3$-interlacing. 
\end{cor}
\begin{proof}
If $I$ and $J$ are $3$-interlacing, then they are both unions of three one-element sets. Hence, all the lateral sides in the boxes from the above proof are of length 1 and the statement follows since $a_i$ from the previous proposition are strictly positive, but at most equal to the lengths of the boxes involved. 
\end{proof}

\begin{cor}\label{cor:ext-1-dim} Let $k=3$. If $I$ and $J$ are crossing but not $3$-interlacing, then ${\rm Ext}^1(L_I,L_J)\cong \mathbb C$. 
\end{cor}

Note that if in Corollary~\ref{cor:ext-1-dim} we have $L_J=\tau(L_I)$, then we are in the situation of Theorem~\ref{Teo rk1-new}.

\begin{prop}\label{prop:rank2-implication}
Assume that 
$(k,n)=(3,9)$ or $(k,n)=(4,8)$. Let $M\in \CM(\B)$ be a rigid indecomposable rank $2$ module. Then $M\cong L_I\mid L_J$ where $I$ and $J$ are $3$-interlacing. 
\end{prop}

\begin{proof}
Let $M=L_I\mid L_J$ be rigid indecomposable, with $I$ and $J$ $r$-interlacing. 
Since $M$ is indecomposable, we 
get $r\in\{3,4\}$. 
If $r=4$, then we must have $k=4$, and the only $4$-interlacing $4$-subsets are $I=\{1,3,5,7\}$ 
and  $J=\{2,4,6,8\}$. Assume that $M=L_I| L_J$ is rigid. Then it has a filtration given by its 
profile. Moreover, if any other module is rigid with the same profile it is isomorphic to $M$. 
On the other hand, if $M$ is rigid, then $\tau(M)$ is also rigid. 
If we compute $\tau^{-1} M = \Omega (M)$ in $\CM (\Pi_{4,8})$, we obtain that they have the 
same filtration so $\tau^{-1} M = M$. So $M$ and $\tau^{-1}(M)$ are the end-terms of an 
Auslander-Reiten-sequence and 
$M$ is not rigid, a contradiction. 
Then $I$ and $J$ must be $3$-interlacing. 
\end{proof}

\begin{cor}\label{cor:real-root3}
Let $M\in \CM(B_{3,9})$ be a rigid indecomposable rank $2$ module. 
Then $\varphi(M)$ is a real root for $J_{3,9}$ of degree $2$. 
\end{cor}

\begin{proof}
By Proposition~\ref{prop:rank2-implication}, $M\cong L_I\mid L_J$ with $I$ and $J$ 3-interlacing, 
so 
$I\cup J$ consists of six distinct elements of $\{1,2,\dots, n\}$. But then 
in $\underline{a}(M)=(a_1,\dots,a_9)$ there are six entries equal to $1$ and three entries equal to 0. 
So $q(\underline{a})=\sum a_i^2 - \frac{1}{9}(\sum a_i)^2=6 - 4=2.$ 
\end{proof}

We observe that there exist rigid rank 2 modules corresponding to imaginary roots, 
an example is $L_{2568}\mid L_{1347}$ (\cite[Figure 13]{JKS16}). We expect that 
if we impose that the modules correspond to real roots, we get a 
counterpart to Proposition~\ref{prop:rank2-implication}. 

\begin{prop}\label{prop:rank2-to-real}
Let $M$ be a rank $2$ indecomposable module with profile $I|J$. If $M$ corresponds to a 
real root of $J_{k,n}$, then $|I\cap J|=k-3.$ 
\end{prop}

\begin{proof} Assume that $|I\cap J|=k-3-m$ for some $m\geq 0$. Note that $|I\cap J|$ cannot 
be greater than $k-3$, because in this case, the corresponding poset would be either $(1^2,2)$ 
or $(1,2)$, which do not correspond to an indecomposable module. If there are $k-3-m$ 
common elements in $I$ and $J$, then the corresponding vector, say $\underline{a}$, has $k-3-m$ 
coordinates equal to 2, $2(m+3)$ coordinates equal to 1, and the rest are equal to 0. If we apply 
our quadratic form $q$ to this vector $\underline{a}$, then we get that $q(\underline{a})=2-2m$. 
In order for $\underline{a}$ to be real, it has to be that $m=0$.
\end{proof}

\begin{ex} Assume $(k,n)=(3,9)$ and $M$ is a rank $2$ indecomposable module with profile $I|J$. 
The conditions on $I$ and $J$ imply that the form $q$ of such a module 
evaluates to 2, as in that case, $\underline{a}=(1,1,1,1,1,1,0,0,0)$, up to permuting 
the entries. 
\end{ex}

\begin{cor}\label{cor:rk2-indec-poset}
Let $M$ be a rank $2$ indecomposable module corresponding to a real root. Then the poset of $M$ is $(1^3,2)$.  
\end{cor}
\begin{proof} If the poset is of the form $(1^r,2)$, where $r\geq 4$, then $|I\cap J |<k-3$. If $r\leq 2$, then the module in 
question is not indecomposable. \end{proof}

It follows that necessary conditions for a rank 2 indecomposable module $M$ with profile $I|J$ to be rigid are $|I\cap J|=k-3$ and that the poset of 
$M$ is $(1^3,2)$. Unfortunately,  we do not know if these conditions are sufficient in the general case. We will show that these 
conditions are sufficient in the tame cases, and we conjecture that it holds in general.  

\begin{notation}\label{notation:tight}
Let $I$ and $J$ be $3$-interlacing $k$-subsets. We say that $I$ and $J$ are 
{\em tightly $3$-interlacing} if $|I\cap J|=k-3$. 
\end{notation}

\begin{lemma}\label{lm:rank-2-ok}  
Let $M\in \CM(\B)$ be a module with $M\cong L_I\mid L_J$ where $I$ and $J$ are tightly $3$-interlacing. 
Then $M$ is indecomposable and $q(M)$ is a real root for $J_{k,n}$. 
\end{lemma}

\begin{proof}
The poset of $M$ is $(1^3,2)$, thus $M$ is indecomposable. Since $|I\cap J|=k-3$, up to permuting the entries, 
$\underline{a}(M)=(\underbrace{2,\dots,2}_{k-3},\underbrace{1,\dots, 1}_6,0,\dots, 0)$, yielding $q(M)=q(a)=2$. 
\end{proof}

\begin{lemma}\label{lm:rank2-constr-ok}
(1) Let $I$ and $J$ be $r$-interlacing, for $r\ge 3$. Then $L(I,J)$ is a rank $2$ module 
in $\CM(\B)$ with filtration $L_I\mid L_J$. \\
(2) If $I$ and $J$ are tightly $3$-interlacing, then $L(I,J)$ and $L(J,I)$ are 
indecomposable. 
\end{lemma} 

Note that Lemma~\ref{lm:rank2-constr-ok} (2) provides modules with cyclically reordered filtration, as 
discussed in~\cite[Observation 8.2]{JKS16}. 

\begin{proof}
(1) 
The module $L_J$ embeds in $L(I,J)$ diagonally 
via the map $a\mapsto (a,a)$ lattice point wise, as in Figure~\ref{fig:lattice-embedding} of 
Example~\ref{ex:construction}, sending column $U_i$ to column $V_i$ in a way to 
get an image of $L_J$ as high up as possible. 
This yields an 
exact sequence $0\to L_J\to L(I,J)\to L_I\to 0$, giving the claimed filtration. 

(2) 
The indecomposability follows from the fact that for tightly 3-interlacing $k$-subsets 
both modules have poset $(1^3,2)$. 
\end{proof}


We expect that tightly 3-interlacing subsets always yield rigid modules. Combined with 
the preceding statements, this would give us: 

\begin{conj}\label{conj:interlacing-ok}
Fix $(k,n)$ with $k\ge 3$. 
Let $I$ and $J$ be tightly $3$-interlacing. Then $L(I,J)$ is a 
rigid indecomposable rank $2$ module. 
\end{conj} 

Theorem~\ref{Teo rk1-new} provides further evidence for 
Conjecture~\ref{conj:interlacing-ok} for arbitrary $(k,n)$: We can use part 1 to find 
many examples of rigid indecomposable rank 2 modules 
$L(X,Y)=L_X\mid L_Y$ where $X$ and $Y$ are tightly 3-interlacing $k$-subsets 
satisfying $|X\cap Y|=k-3$ by choosing $j=i+3$ in the theorem. 

\begin{conj}\label{conj:exclude-4}
Let $M$ be a rank $2$ module with poset $(1^r,2)$, for $r\ge 4$. Then $M$ is not rigid. 
\end{conj}

We recall that in the finite cases with $k=3$, the numbers of (rigid) indecomposable rank $2$ 
modules are 2 (for $n=6$), 14 (for $n=7$) and 56 (for $n=8$). All these correspond to real roots for the associated 
Kac--Moody algebra $J_{k,n}$. 

In the tame cases $(3,9)$ and $(4,8)$, we will show that there are, respectively,  168 and 120 rigid indecomposable modules of rank 2.  
This follows from Proposition~\ref{prop:rank2-implication} and the fact 
that in these cases, the real root in question  has to correspond to a $9$-tuple $(a_1,\dots, a_9)$ with 
six entries equal to  1 and zeros elsewhere, or to an $8$-tuple $(a_1,\dots, a_8)$ with one entry 
equal to 2, six entries equal to 1 and one 0 respectively.  We will confirm the above numbers 
explicitly in the next two sections  by computing all tubes that contain 
rank 2 modules in the Auslander-Reiten quiver.

This leads us to a conjectured formula for the number of rigid indecomposable rank 2 modules 
which correspond to 
real roots.  
The 3-interlacing property (Proposition~\ref{prop:rank2-implication}) yields the factor ${n\choose 6}$, 
a choice of 6 elements from $[1,n]$, say $\{1\le i_1<j_1<i_2<j_2<i_3<j_3\le n\}$, 
with $\{i_1,i_2,i_3\}\subset I$ and 
$\{j_1,j_2,j_3\}\subset J$. 
If Conjecture~\ref{conj:interlacing-ok} is true, each pair $I$ and $J$ of 3-interlacing subsets 
where the 
remaining $k-3$ labels are common to $I$ and $J$ yields 
two rigid indecomposable rank 2 modules. Using the map from indecomposable modules to roots for 
$J_{k,n}$ (see Section~\ref{sec:roots}) we see that these give rise to real roots. 
So there is a choice of $k-3$ elements from the remaining $n-6$ elements of $[1,n]$, yielding a factor ${n-6\choose k-3}$. 
Finally, there is a factor 2 which arises from the choice 
of which of these subsets is $I$ and which is $J$. 
The above arguments give an upper  
bound for the number of rigid indecomposable rank $2$ modules corresponding 
to real roots.

\begin{conj}\label{conj:counting}
Let $3\le k\le n/2$. 
For every real root $\alpha$ of degree $2$ there are exactly two non-isomorphic rigid indecomposable rank $2$ modules 
$M_1$ and $M_2$ such that $\varphi(M_1)=\varphi(M_2)=\alpha$. 
Thus, there are  $2{n \choose 6}{n-6\choose k-3}$ rigid indecomposable rank $2$ modules 
corresponding to real roots. 
\end{conj}


\begin{rem} 
Note that we have proved (Proposition~\ref{prop:rank2-to-real}, 
Corollary~\ref{cor:rk2-indec-poset}) 
that the number $2{n \choose 6}{n-6\choose k-3}$ is the number of indecomposable 
rank 2 modules corresponding to real roots. Combinatorially, indecomposable rank 2 
modules corresponding to real roots are determined by the conditions $|I\cap J|=k-3$ 
(which follows from the quadratic form) and the poset of the module is $(1^3,2)$ (which 
follows from the indecomposability requirement).
\end{rem}

%
\section{The tame case $(k,n)=(3,9)$}\label{sec:3-9}

Here we will describe all the tubes of $\CM(B_{3,9})$ which contain rank 1 and rank 2 modules. 
Recall that every exceptional tube of rank $s$ contains $s-1$ $\, $  $\tau$-orbits of rigid modules. We will 
write down the $\tau$-orbits where we can give explicit filtrations. 
In the end, we use this to find the number of rigid indecomposable rank 2 modules in 
$\CM(B_{3,9})$. 
We start by describing all such tubes of rank 6 in Figure~\ref{fig:rank6}. 
For example, Figure~\ref{fig:rank6}(A) shows one of the three tubes containing projective-injective modules, 
Figure~\ref{fig:rank6} (B) shows one of the three tubes with modules of the form 
$L_{i,i+1,i+6}$.

\begin{figure}[H]
\begin{subfigure}{0.4\textwidth}
$
\xymatrix@=0.4em{
789\ar@{--}[dd]\ar@{-}[rd] &&&& 123\ar@{-}[rd]   \\
\ar@{.}[r]& 178\ar@{-}[rd] \ar@{.}[rr] && 923\ar@{.}[rr]\ar@{-}[ru]\ar@{-}[rd]  && \dots \\ 
168\ar@{--}[dd]\ar@{.}[rr]\ar@{-}[rd]\ar@{-}[ru]  && \frac{279}{138}\ar@{.}[rr]\ar@{-}[rd]\ar@{-}[ru] && 
249\ar@{.}[r]\ar@{-}[ru]\ar@{-}[rd]   &   \\ 
\ar@{.}[r]  & \frac{269}{138}\ar@{.}[rr]\ar@{-}[rd]\ar@{-}[ru] && 
 \frac{279}{148}\ar@{.}[rr]\ar@{-}[rd]\ar@{-}[ru]  && \dots \\
 {\tiny\thfrac{469}{582}{713}} \ar@{.}[rr]\ar@{-}[ru]  && 
  \frac{269}{148} \ar@{.}[rr]\ar@{-}[ru]   && {\tiny\thfrac{379}{825}{146}}\ar@{.}[r]\ar@{-}[ru] &  
}
$
\caption{Tube with projective-injectives}
\end{subfigure}
\begin{subfigure}{0.4\textwidth}
$
\xymatrix@=0.4em{
 145\ar@{.}[rr]\ar@{--}[d]\ar@{-}[rd] 
  && 236\ar@{.}[rr]\ar@{-}[rd]  && 478   \ar@{-}[rd] \ar@{.}[r] & \\
\ar@{.}[r] &\frac{246}{135}\ar@{.}[rr]\ar@{-}[ru]   && 
 \frac{247}{368}\ar@{.}[rr]\ar@{-}[ru] && \dots 
}
$
\caption{Tube with $L_{145}$}
\end{subfigure}
%
\vskip .5cm
%
\begin{subfigure}{0.4\textwidth}
$
\xymatrix@=0.4em{
136\ar@{.}[rr]\ar@{--}[d]\ar@{-}[rd] && \frac{247}{358} \ar@{.}[rr]\ar@{-}[rd] 
 && 469 \ar@{.}[r]\ar@{-}[rd] & \\ 
 \ar@{.}[r] &\rkk 3 \ar@{-}[ru]\ar@{.}[rr]&& \rkk 3\ar@{-}[ru]\ar@{.}[rr] && \dots
}
$
\caption{Tube with $L_{136}$}
\end{subfigure}
\begin{subfigure}{0.4\textwidth}
$
\xymatrix@=0.4em{
137\ar@{.}[rr]\ar@{--}[d]\ar@{-}[rd]  && \frac{248}{359} \ar@{.}[rr]\ar@{-}[rd] 
 && 146 \ar@{.}[r]\ar@{-}[rd] & \\ 
 \ar@{.}[r]&\rkk 3 \ar@{-}[ru]\ar@{.}[rr] && \rkk 3 \ar@{-}[ru]\ar@{.}[rr] && \dots
}
$
\caption{Tube with $L_{137}$}
\end{subfigure}
%
\vskip .5cm
%
\begin{subfigure}{0.4\textwidth}
$
\xymatrix@=0.4em{
\frac{469}{357}\ar@{.}[rr]\ar@{--}[d]\ar@{-}[rd] && \frac{146}{258} \ar@{.}[rr]\ar@{-}[rd] 
 && \frac{379}{168} \ar@{.}[r]\ar@{-}[rd] &   \\
\ar@{.}[r]& \rkk 4\ar@{.}[rr]\ar@{-}[ru] && \rkk 4\ar@{.}[rr]\ar@{-}[ru] 
 && \dots 
}
$
\caption{Tube with $L_{469}\mid L_{357}$}
\end{subfigure}
\begin{subfigure}{0.4\textwidth}
$
\xymatrix@=0.4em{
\frac{579}{368}\ar@{.}[rr]\ar@{--}[d]\ar@{-}[rd] && \frac{147}{259} \ar@{.}[rr]\ar@{-}[rd] 
 && \frac{138}{269} \ar@{.}[r]\ar@{-}[rd] &   \\
\ar@{.}[r]& \rkk 4\ar@{.}[rr]\ar@{-}[ru] && \rkk 4\ar@{.}[rr]\ar@{-}[ru] 
 &&  \dots
}
$
\caption{Tube with $L_{579}\mid L_{368}$}
\end{subfigure}
%
\vskip .5cm
%
\begin{subfigure}{0.4\textwidth}
$
\xymatrix@=0.4em{
\frac{258}{146}\ar@{.}[rr]\ar@{--}[d]\ar@{-}[rd] && {\tiny\thfrac{359}{247}{136}} \ar@{.}[rr]\ar@{-}[rd]  && 
  \frac{258}{479}\ar@{-}[rd]  \ar@{.}[r]  &  \\
\ar@{.}[r]& \rkk 5\ar@{.}[rr]\ar@{-}[ru] && \rkk 5\ar@{.}[rr]\ar@{-}[ru] && \dots   
}
$
\caption{Tube with $L_{258}\mid L_{146}$}
\end{subfigure}
\begin{subfigure}{0.4\textwidth}
$
\xymatrix@=0.4em{
\frac{257}{146}\ar@{.}[rr]\ar@{--}[d]\ar@{-}[rd] && {\tiny\thfrac{358}{247}{369}} \ar@{.}[rr]\ar@{-}[rd] && 
 \frac{158}{479}\ar@{-}[rd] \ar@{.}[r]  &   \\ 
\ar@{.}[r]& \rkk 5\ar@{.}[rr]\ar@{-}[ru] && \rkk 5\ar@{.}[rr]\ar@{-}[ru] && \dots 
}
$
\caption{Tube with $L_{257}\mid L_{146}$}
\end{subfigure}
%
\vskip .5cm
%
\begin{subfigure}{0.4\textwidth}
$
\xymatrix@=0.4em{
\frac{268}{147}\ar@{.}[rr]\ar@{--}[d]\ar@{-}[rd] && {\tiny\thfrac{359}{248}{136}} \ar@{.}[rr]\ar@{-}[rd] 
 && \frac{259}{147}  \ar@{-}[rd] \ar@{.}[r] &     \\
\ar@{.}[r]& \rkk 5\ar@{.}[rr]\ar@{-}[ru] && \rkk 5\ar@{.}[rr]\ar@{-}[ru] 
 & & \dots
}
$
\caption{Tube with $L_{268}\mid L_{147}$}
\end{subfigure}
\begin{subfigure}{0.4\textwidth}
$
\xymatrix@=0.4em{
\frac{247}{136}\ar@{.}[rr]\ar@{--}[d]\ar@{-}[rd] && {\tiny\thfrac{258}{479}{368}} \ar@{.}[rr]\ar@{-}[rd] 
 && \frac{157}{469} \ar@{.}[r]\ar@{-}[rd]  &    \\
\ar@{.}[r]& \rkk 5\ar@{.}[rr]\ar@{-}[ru] && \rkk 5\ar@{.}[rr]\ar@{-}[ru] && \dots 
}
$
\caption{Tube with $L_{247}\mid L_{136}$}
\end{subfigure}
\caption{Rank 6 tubes for $\CM(B_{3,9})$}\label{fig:rank6}
\end{figure}
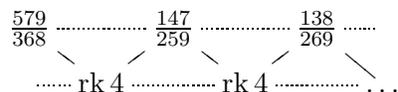
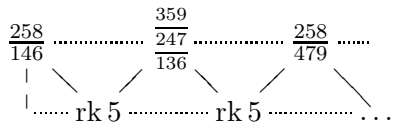
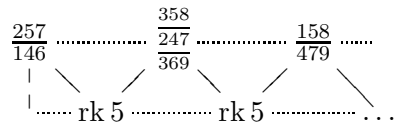

Next we give the rank 3 tubes containing rigid modules in Figure~\ref{fig:tube-R3}. There are two such types 
(three tubes of this form). 

\begin{figure}[H]
\begin{subfigure} {0.4\textwidth}
$
\xymatrix@=0.4em{
126\ar@{.}[rr]\ar@{--}[d]\ar@{-}[rd]   && 378\ar@{.}[rr]\ar@{-}[rd]  
 && 459\ar@{.}[rr]\ar@{-}[rd] && 126\ar@{--}[d]  \\
\ar@{.}[r] & \frac{137}{268} \ar@{.}[rr]\ar@{-}[ru]  && 
 \frac{479}{358}\ar@{.}[rr]\ar@{-}[ru]  && 
 \frac{146}{259}\ar@{.}[r]\ar@{-}[ru]  &  
}
$
\caption{Tube with $L_{126}$}
\end{subfigure}
\begin{subfigure}{0.4\textwidth}
$
\xymatrix@=0.4em{
\frac{358}{146}\ar@{.}[rr]\ar@{--}[dd]\ar@{-}[rd]   && \frac{259}{137}\ar@{.}[rr]\ar@{-}[rd]  
 && \frac{268}{479}\ar@{.}[rr]\ar@{-}[rd] && \frac{358}{146}\ar@{--}[dd]  \\
\ar@{.}[r] & \rkk 4 \ar@{.}[rr]\ar@{-}[ru] && 
 \rkk 4\ar@{.}[rr]\ar@{-}[ru] && 
\rkk 4\ar@{.}[r]\ar@{-}[ru] & \\
 && && && 
}
$
\caption{Tube with $L_{358}\mid L_{146}$}
\end{subfigure}
\caption{Rank 3 tubes for $\CM(B_{3,9})$} \label{fig:tube-R3} 
\end{figure}
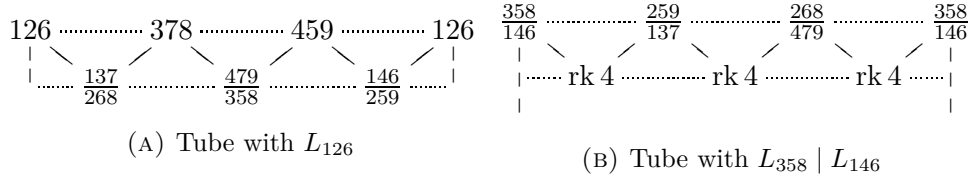

Furthermore, there are two types of tubes of rank 2, described in Figure~\ref{fig:tube-R2}. The modules in the second 
$\tau$-orbit are not rigid. To indicate this, they are written in grey. 

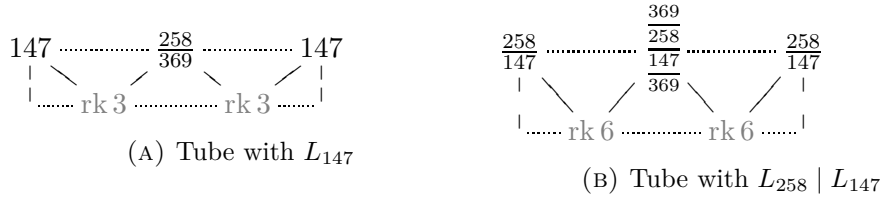
\begin{figure}[H]
\begin{subfigure}{0.4\textwidth}
$
\xymatrix@=0.4em{
147 \ar@{.}[rr]\ar@{--}[d]\ar@{-}[rd] && \frac{258}{369}  \ar@{.}[rr]\ar@{-}[rd] && 147\ar@{--}[d] \\
\ar@{.}[r]  & \textcolor{gray}{ \rkk 3}\ar@{-}[ru]\ar@{.}[rr] && \textcolor{gray}{ \rkk 3}\ar@{-}[ru]\ar@{.}[r] &&& 
}
$
\caption{Tube with $L_{147}$}
\end{subfigure}
\begin{subfigure}{0.4\textwidth}
$
\xymatrix@=0.4em{
\frac{258}{147}\ar@{.}[rr]\ar@{--}[d]\ar@{-}[rd]   &&{\tiny\ffrac{369}{258}{147}{369}} \ar@{.}[rr]\ar@{-}[rd]  
 &&\frac{258}{147}\ar@{--}[d]\\
\ar@{.}[r] &\textcolor{gray}{ \rkk 6} \ar@{.}[rr]\ar@{-}[ru]& &
\textcolor{gray}{  \rkk 6}\ar@{.}[r]\ar@{-}[ru]& 
}
$
\caption{Tube with $L_{258}\mid L_{147}$}
\end{subfigure}
\caption{Rank 2 tubes for $\CM(B_{3,9})$}\label{fig:tube-R2}
\end{figure}

%
\subsection{Counting rigid indecomposables for $Gr(3,9)$}\label{sec:count-all}

$\ $

\noindent By Proposition~\ref{prop:rank2-implication} there are at most 168 rigid indecomposable rank 2 modules in $\CM(B_{3,9})$, because there are 84 3-interlacing pairs 
$(I,J)$ in that case. The tubes with rank 1 modules contain 84 rigid rank 2 modules.  
The remaining 84 rank 2 modules are at the mouths of further tubes of rank 2,3 and 6. 
%
To sum up, the number of rigid indecomposable modules of rank 1, 2 and 3 in the above tubes, listed by rank is: 
\[
\begin{tabular}{|l|ccc|c}
\hline
rank & 1 & 2 & 3 \\
\hline 
$\,\,\,\#$ & 84 & 168 & 117   \\
\hline 
\end{tabular}
\]
We have collected all rank 1 and rank 2 rigid indecomposable modules. Since there are 168 rigid indecomposable modules 
whose roots are real, and 84 real roots of $J_{3,9}$ of degree 2, we have proved that for every real root of degree 2 there are two rigid indecomposable modules $L_J|L_I$ and $L_I|L_J$. Moreover in this case the number of rigid indecomposables of rank 2 is \emph{exactly} twice the number of real roots of degree 2.

Therefore,  both Theorem \ref{te:cyclemods} and Conjecture \ref{conj:counting} hold in this case.

%
\section{The tame case $(k,n)=(4,8)$}\label{sec:4-8}

Here we will describe all the tubes of $\CM(B_{4,8})$ which contain rank 1 and rank 2 modules, 
as in the previous section for $(k,n)=(3,9)$. We use this to find the number of rigid indecomposable 
rank 2 modules in this case. 
We start by describing all such tubes of rank 4 in Figure~\ref{fig:rank4}. 
For example, Figure~\ref{fig:rank4}(A) shows one of the four tubes containing projective-injective modules, 
Figure~\ref{fig:rank4} (B) shows one of the four tubes with rank 1 modules  for $4$-subsets of the form $I=\{i,i+1,i+2,i+5\}$.   

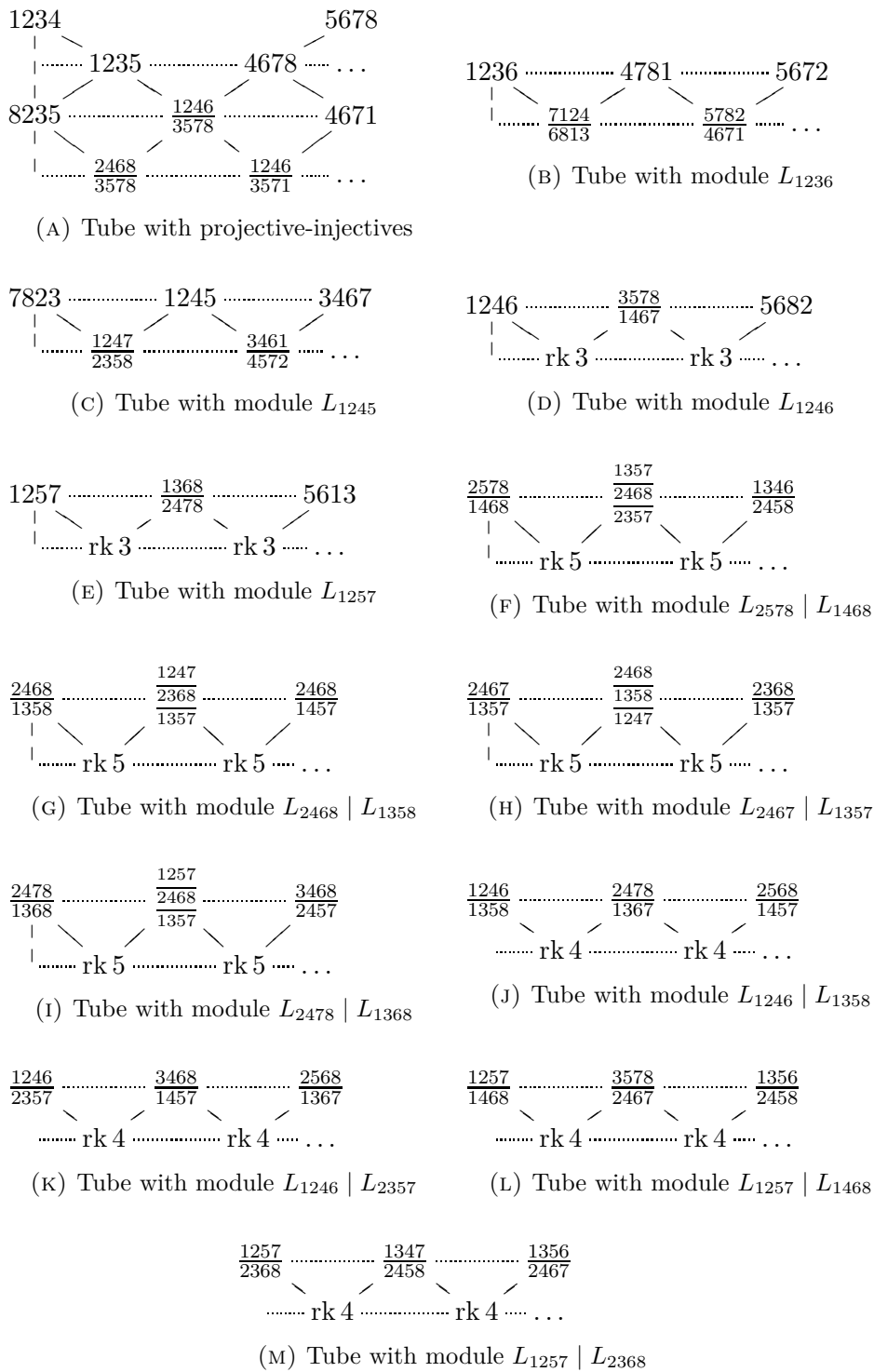
\begin{figure}[H]
\begin{subfigure}{0.4\textwidth}
$
\xymatrix@=0.4em{
 1234\ar@{--}[dd]\ar@{-}[rd] &&  && 5678     \\
  \ar@{.}[r]& 1235 \ar@{.}[rr]\ar@{-}[rd] && 4678 \ar@{.}[r]\ar@{-}[rd]\ar@{-}[ru] & \dots\\
   8235\ar@{--}[d]\ar@{.}[rr]\ar@{-}[rd]\ar@{-}[ru] && 
    \frac{1246}{3578} \ar@{.}[rr]\ar@{-}[rd]\ar@{-}[ru] && 
     4671 \\
  \ar@{.}[r]  &  \frac{2468}{3578} \ar@{.}[rr]\ar@{-}[ru] && 
   \frac{1246}{3571}\ar@{.}[r]\ar@{-}[ru] & \dots 
}
$
\caption{Tube with projective-injectives}
\end{subfigure}
\begin{subfigure}{0.4\textwidth}
$
\xymatrix@=0.4em{
1236\ar@{.}[rr]\ar@{--}[d]\ar@{-}[rd]  && 4781\ar@{.}[rr]\ar@{-}[rd] 
 && 5672\  \\
\ar@{.}[r] & {\frac{7124}{6813}} \ar@{.}[rr]\ar@{-}[ru] 
 &&  {\frac{5782}{4671}} \ar@{.}[r] \ar@{-}[ru]  & \dots 
}
$
\caption{Tube with module $L_{1236}$}
\end{subfigure}
%
\vskip .5cm
%
\begin{subfigure}{0.4\textwidth}
$
\xymatrix@=0.4em{
7823\ar@{--}[d]\ar@{.}[rr]\ar@{-}[rd]  && 1245\ar@{.}[rr]\ar@{-}[rd] 
 && 3467   \\
\ar@{.}[r] & {\frac{1247}{2358}}\ar@{.}[rr] \ar@{-}[ru]  && {\frac{3461}{4572}} \ar@{.}[r] \ar@{-}[ru]  & \dots   
}
$
\caption{Tube with module $L_{1245}$}
\end{subfigure}
\begin{subfigure}{0.4\textwidth}
$
\xymatrix@=0.4em{
1246\ar@{--}[d]\ar@{.}[rr]\ar@{-}[rd]   && \frac{3578}{1467}\ar@{.}[rr]\ar@{-}[rd]  
 && 5682  \\
\ar@{.}[r] & \rkk 3\ar@{.}[rr]\ar@{-}[ru] 
 && \rkk 3\ar@{.}[r]\ar@{-}[ru] & \dots    \\
}
$
\caption{Tube with module $L_{1246}$}
\end{subfigure}
%
\vskip .5cm
%
\begin{subfigure}{0.4\textwidth}
$
\xymatrix@=0.4em{
1257\ar@{--}[d]\ar@{.}[rr]\ar@{-}[rd]   && \frac{1368}{2478}\ar@{.}[rr]\ar@{-}[rd]  
 && 5613  \\
\ar@{.}[r] & \rkk 3\ar@{.}[rr]\ar@{-}[ru] 
 && \rkk 3\ar@{.}[r]\ar@{-}[ru] &   \dots \\
}
$
\caption{Tube with module $L_{1257}$}
\end{subfigure}
\begin{subfigure}{0.4\textwidth}
$
\xymatrix@=0.4em{
\frac{2578}{1468}\ar@{--}[d]\ar@{.}[rr]\ar@{-}[rd]   && {\tiny\thfrac{1357}{2468}{2357}}\ar@{.}[rr]\ar@{-}[rd]  
 && \frac{1346}{2458}  \\
\ar@{.}[r] & \rkk 5\ar@{.}[rr]\ar@{-}[ru] 
 && \rkk 5\ar@{.}[r]\ar@{-}[ru]  &  \dots   
}
$
\caption{Tube with module $L_{2578}\mid L_{1468}$}
\end{subfigure}
%
\vskip .5cm
%
\begin{subfigure}{0.4\textwidth}
$
\xymatrix@=0.4em{
\frac{2468}{1358}\ar@{--}[d]\ar@{.}[rr]\ar@{-}[rd]   && {\tiny\thfrac{1247}{2368}{1357}}\ar@{.}[rr]\ar@{-}[rd]  
 && \frac{2468}{1457} \\
\ar@{.}[r] & \rkk 5\ar@{.}[rr]\ar@{-}[ru] 
 && \rkk 5\ar@{.}[r]\ar@{-}[ru] &   \dots \\
}
$
\caption{Tube with module $L_{2468}\mid L_{1358}$}
\end{subfigure}
\begin{subfigure}{0.4\textwidth}
$
\xymatrix@=0.4em{
\frac{2467}{1357}\ar@{--}[d]\ar@{.}[rr]\ar@{-}[rd]   && {\tiny\thfrac{2468}{1358}{1247}}\ar@{.}[rr]\ar@{-}[rd]  
 && \frac{2368}{1357}  \\
\ar@{.}[r] & \rkk 5\ar@{.}[rr]\ar@{-}[ru] 
 && \rkk 5\ar@{.}[r]\ar@{-}[ru]  &  \dots   
}
$
\caption{Tube with module $L_{2467}\mid L_{1357}$}
\end{subfigure}
%
\vskip .5cm
%
\begin{subfigure}{0.4\textwidth}
$
\xymatrix@=0.4em{
\frac{2478}{1368}\ar@{--}[d]\ar@{.}[rr]\ar@{-}[rd]   && {\tiny\thfrac{1257}{2468}{1357}}\ar@{.}[rr]\ar@{-}[rd]  
 && \frac{3468}{2457} \\
\ar@{.}[r] & \rkk 5\ar@{.}[rr]\ar@{-}[ru] 
 && \rkk 5\ar@{.}[r]\ar@{-}[ru] &   \dots \\
}
$
\caption{Tube with module $L_{2478}\mid L_{1368}$}
\end{subfigure}
\begin{subfigure}{0.4\textwidth}
$
\xymatrix@=0.4em{
\frac{1246}{1358}\ar@{--}[d]\ar@{.}[rr]\ar@{-}[rd]   && \frac{2478}{1367}\ar@{.}[rr]\ar@{-}[rd]  
 && \frac{2568}{1457}  \\
\ar@{.}[r] & \rkk 4\ar@{.}[rr]\ar@{-}[ru] 
 && \rkk 4\ar@{.}[r]\ar@{-}[ru]  &  \dots   
}
$
\caption{Tube with module $L_{1246}\mid L_{1358}$}
\end{subfigure}
%
\vskip .5cm
%
\begin{subfigure}{0.4\textwidth}
$
\xymatrix@=0.4em{
\frac{1246}{2357}\ar@{--}[d]\ar@{.}[rr]\ar@{-}[rd]   && \frac{3468}{1457}\ar@{.}[rr]\ar@{-}[rd]  
 & & \frac{2568}{1367} \\
\ar@{.}[r] & \rkk 4\ar@{.}[rr]\ar@{-}[ru] 
 && \rkk 4\ar@{.}[r]\ar@{-}[ru] &   \dots \\
}
$
\caption{Tube with module $L_{1246}\mid L_{2357}$}
\end{subfigure}
\begin{subfigure}{0.4\textwidth}
$
\xymatrix@=0.4em{
\frac{1257}{1468}\ar@{--}[d]\ar@{.}[rr]\ar@{-}[rd]   && \frac{3578}{2467}\ar@{.}[rr]\ar@{-}[rd]  
 && \frac{1356}{2458}  \\
\ar@{.}[r] & \rkk 4\ar@{.}[rr]\ar@{-}[ru] 
 && \rkk 4\ar@{.}[r]\ar@{-}[ru]  &  \dots   
}
$
\caption{Tube with module $L_{1257}\mid L_{1468}$}
\end{subfigure}
\vskip .5cm
\begin{subfigure}{0.4\textwidth}
$
\xymatrix@=0.4em{
\frac{1257}{2368}\ar@{--}[d]\ar@{.}[rr]\ar@{-}[rd]   && \frac{1347}{2458}\ar@{.}[rr]\ar@{-}[rd]  
 && \frac{1356}{2467}  \\
\ar@{.}[r] & \rkk 4\ar@{.}[rr]\ar@{-}[ru] 
 && \rkk 4\ar@{.}[r]\ar@{-}[ru]  &  \dots   
}
$
\caption{Tube with module $L_{1257}\mid L_{2368}$}
\end{subfigure}
\caption{Rank four tubes for $\CM(B_{4,8})$.}\label{fig:rank4}
\end{figure}

In addition to these, there are two types of rank 2 tubes in $\CM(B_{4,8})$. They are described in Figure~\ref{fig:rank2}. 
As before, grey indicates non-rigid modules.

\begin{figure}[H]
\begin{subfigure}{0.4\textwidth}
$
\xymatrix@=0.4em{
1256\ar@{--}[dd]\ar@{.}[rr] \ar@{-}[dr]  && 3478\ar@{-}[dr]\ar@{.}[rr]  && 1256 \ar@{--}[dd]  \\
\ar@{.}[r] &\textcolor{gray}{ \rkk 2}\ar@{-}[ur] \ar@{.}[rr]  &&\textcolor{gray}{ \rkk 2} \ar@{-}[ur] \ar@{.}[r]  &   \\
 && && && 
}
$
\caption{Tube with module $L_{1256}$}
\end{subfigure}
\begin{subfigure}{0.4\textwidth}
$
\xymatrix@=0.4em{
1357\ar@{--}[d]\ar@{.}[rr] \ar@{-}[rd]  && {\tiny\thfrac{2468}{1357}{2468}}\ar@{.}[rr] \ar@{-}[rd]  && 1357 \ar@{--}[d]  \\
\ar@{.}[r] & \textcolor{gray}{ \rkk 4} \ar@{-}[ur]\ar@{.}[rr]   && \textcolor{gray}{ \rkk 4} \ar@{-}[ur]\ar@{.}[r] &
}
$
\caption{Tube with module $L_{1357}$}
\end{subfigure}
\caption{Rank 2 tubes for $\CM(B_{4,8})$}\label{fig:rank2}
\end{figure}
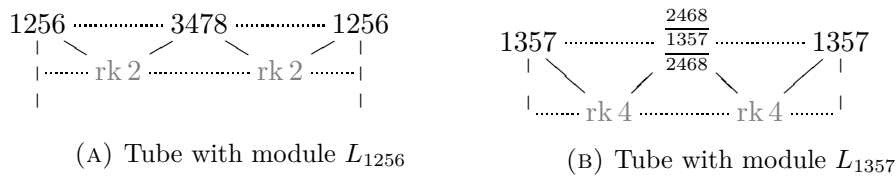

%
\subsection{Summing up}\label{sec:count-all-48}

To sum up, the number of  rigid indecomposable modules of rank 1,2 and 3 in all these tubes is: 
\[
\begin{tabular}{|l|ccc|c}
\hline
rank & 1 & 2 & 3 \\
\hline 
$\,\,\,\#$ & 70 & 120 & 82   \\
\hline 
\end{tabular}
\]
Since there are 70 rank 1 modules for $(4,8)$, we have covered all tubes containing such modules. 
Overall, there are 120 rank 2 rigid indecomposable modules. Among these, there are eight modules that do not correspond to the real roots. These are the
modules with profile of the form $1246|3578$. So there are 112 rigid indecomposable modules of rank 2 that correspond to real roots of $J_{4,8}$. 
Since there are 56 roots of degree 2, we have shown that the number of rigid indecomposable rank 2 modules corresponding to real roots 
is twice the number of real roots of degree 2 for $J_{4,8}$. 
Also, for every rank 2 rigid indecomposable module whose root is real, and whose filtration is $L_I|L_J$, 
there exist a rigid indecomposable module with filtration $L_J|L_I$. Therefore, 
both Theorem \ref{te:cyclemods} and Conjecture \ref{conj:counting} hold in this case.

\section*{Acknowledgements} 

We thank Osamu Iyama, Bernhard Keller and Alastair King 
for very useful discussions. 
In particular, we thank Matthew Pressland for many discussions and for very 
helpful comments 
on a first version of this work. 
The authors were supported by the Austrian Science Fund Project Number P25647-N26 
(K.B. and A.G.E.), Number DK W1230 (K.B.), 
and Austrian Science Fund Project Number P29807-N35 (D.B.).

\bibliographystyle{alpha}
\bibliography{biblio}

\begin{thebibliography}{RVdB02}

\bibitem[Ami09]{Amiot2009}
Claire Amiot.
\newblock Cluster categories for algebras of global dimension 2 and quivers
  with potential.
\newblock {\em Ann. Inst. Fourier (Grenoble)}, 59(6):2525--2590, 2009.

\bibitem[BB16]{BB}
Karin Baur and Dusko Bogdanic.
\newblock Extensions between {C}ohen--{M}acaulay modules of {G}rassmannian
  cluster categories.
\newblock {\em Journal of Algebraic Combinatorics}, pages 1--36, 2016.

\bibitem[BKL10]{Barot}
Michael Barot, Dirk Kussin, and Helmut Lenzing.
\newblock The cluster category of a canonical algebra.
\newblock {\em Transactions of the American Mathematical Society},
  362(8):4313--4330, 2010.

\bibitem[BKM16]{BKM16}
Karin Baur, Alastair~D. King, and Robert~J. Marsh.
\newblock Dimer models and cluster categories of {G}rassmannians.
\newblock {\em Proc. Lond. Math. Soc. (3)}, 113(2):213--260, 2016.

\bibitem[Buc86]{Buc}
Ragnar~O. Buchweitz.
\newblock Maximal {C}ohen-{M}acaulay modules and {T}ate-cohomology over
  {G}orenstein rings.
\newblock {\em University of Hannover}, -, 1986.

\bibitem[DL16]{DL16}
Laurent Demonet and Xueyu Luo.
\newblock Ice quivers with potential associated with triangulations and
  {C}ohen-{M}acaulay modules over orders.
\newblock {\em Trans. Amer. Math. Soc.}, 368(6):4257--4293, 2016.

\bibitem[DWZ08]{DWZ08}
Harm Derksen, Jerzy Weyman, and Andrei Zelevinsky.
\newblock Quivers with potentials and their representations. {I}. {M}utations.
\newblock {\em Selecta Math. (N.S.)}, 14(1):59--119, 2008.

\bibitem[GGS15]{GG15}
Christof Geiss and Ra{\'u}l Gonz{\'a}lez-Silva.
\newblock Tubular jacobian algebras.
\newblock {\em Algebras and Representation Theory}, 18(1):161--181, 2015.

\bibitem[GLFS16]{GLab}
Christof Gei{\ss}, Daniel Labardini-Fragoso, and Jan Schr{\"o}er.
\newblock The representation type of {J}acobian algebras.
\newblock {\em Advances in Mathematics}, 290:364--452, 2016.

\bibitem[GLS06]{rigid}
Christof Gei\ss, Bernard Leclerc, and Jan Schr\"oer.
\newblock Rigid modules over preprojective algebras.
\newblock {\em Invent. Math.}, 165(3):589--632, 2006.

\bibitem[GLS08]{GLS08}
Christof Geiss, Bernard Leclerc, and Jan Schr\"oer.
\newblock Partial flag varieties and preprojective algebras.
\newblock {\em Ann. Inst. Fourier (Grenoble)}, 58(3):825--876, 2008.

\bibitem[Hap88]{H}
Dieter Happel.
\newblock {\em Triangulated categories in the representation theory of
  finite-dimensional algebras}, volume 119 of {\em London Mathematical Society
  Lecture Note Series}.
\newblock Cambridge University Press, Cambridge, 1988.

\bibitem[JKS16]{JKS16}
Bernt~Tore Jensen, Alastair~D. King, and Xiuping Su.
\newblock A categorification of {G}rassmannian cluster algebras.
\newblock {\em Proc. Lond. Math. Soc. (3)}, 113(2):185--212, 2016.

\bibitem[Kac90]{Kac90}
Victor~G. Kac.
\newblock {\em Infinite-dimensional {L}ie algebras}.
\newblock Cambridge University Press, Cambridge, third edition, 1990.

\bibitem[Kel13]{Keller2013}
Bernhard Keller.
\newblock The periodicity conjecture for pairs of {D}ynkin diagrams.
\newblock {\em Ann. of Math. (2)}, 177(1):111--170, 2013.

\bibitem[Knu]{KK}
Allen Knutson.
\newblock \url{https://plus.google.com/+AllenKnutson/posts/TWsWhCakCQg}.
\newblock Accessed: 2018-04-28.

\bibitem[KR07]{KellerReiten07}
Bernhard Keller and Idun Reiten.
\newblock Cluster-tilted algebras are {G}orenstein and stably {C}alabi-{Y}au.
\newblock {\em Adv. Math.}, 211(1):123--151, 2007.

\bibitem[{Pos}06]{Postnikov}
A.~{Postnikov}.
\newblock {Total positivity, {G}rassmannians, and networks}.
\newblock {\em ArXiv Mathematics e-prints}, September 2006.

\bibitem[RVdB02]{reiten2002noetherian}
Idun Reiten and Michel Van~den Bergh.
\newblock Noetherian hereditary abelian categories satisfying {S}erre duality.
\newblock {\em Journal of the American Mathematical Society}, 15(2):295--366,
  2002.

\bibitem[Sco06]{Scott06}
Joshua~S. Scott.
\newblock Grassmannians and cluster algebras.
\newblock {\em Proc. London Math. Soc. (3)}, 92(2):345--380, 2006.

\bibitem[Sim93]{S93}
Daniel Simson.
\newblock {\em Linear representations of partially ordered sets and vector
  space categories}, volume~4.
\newblock CRC Press, 1993.

\bibitem[Zel12]{Zel-conj}
Andrei Zelevinsky.
\newblock Private communcation, {Z}urich, 2012.

\end{thebibliography}

\end{document}